\documentclass[a4paper,11pt]{amsart}
\usepackage{amsfonts,bm}
\usepackage{amssymb}
\usepackage[utf8]{inputenc}
\usepackage{amsmath}
\usepackage{pdflscape}
\usepackage{float}
\usepackage{tikz}
\usetikzlibrary{arrows.meta,positioning}
\usepackage{easyReview}
\usepackage{graphicx}
\usepackage[colorlinks=true, linkcolor=red, citecolor=blue]{hyperref}

\usepackage[]{epsfig}
\usepackage[]{pstricks}
\usepackage{tikz}

\newtheorem{theorem}{Theorem}[section]
\newtheorem{proposition}[theorem]{Proposition}
\newtheorem{lemma}[theorem]{Lemma}
\newtheorem{corollary}[theorem]{Corollary}

\newtheorem{remark}{Remark}
\theoremstyle{remark}

\usepackage{mathrsfs}

\numberwithin{equation}{section}
\makeatletter
\@namedef{subjclassname@2010}{\textup{2020} Mathematics Subject Classification}
\makeatother

\begin{document}
	
	\pagenumbering{arabic}	
	\title[Linearized Whitham-Boussinesq System]{Optimal Polynomial Stabilization of the Linearized Periodic Whitham--Boussinesq System}
	\author[Capistrano-Filho]{Roberto de A. Capistrano-Filho*}
	\author[Roqueta]{William Artiles Roqueta}
	\address{Departamento de Matem\'atica,  Universidade Federal de Pernambuco (UFPE), 50740-545, Recife (PE), Brazil.}
	\email{\url{roberto.capistranofilho@ufpe.br}}
	\email{\url{william.roqueta@ufpe.br}}

		\thanks{*Corresponding author: \url{roberto.capistranofilho@ufpe.br}}
		

\date{\today}
	
\begin{abstract}
We study the stabilization of the linearized periodic Whitham--Boussinesq system on the one-dimensional torus. We establish the well-posedness of the conservative and damped dynamics in the natural energy space and describe the spectral structure of the conservative generator, whose frequencies exhibit sublinear growth of order $|k|^{1/2}$ at high frequency. We then prove strong stability of the damped semigroup and obtain a high-frequency resolvent estimate with linear growth in the spectral parameter. Under genuinely localized damping, a family of high-frequency quasimodes provides the matching lower bound and shows that this resolvent growth is optimal. By the Borichev--Tomilov theorem \cite{BorichevTomilov2010}, we deduce a $t^{-1}$ decay rate for the semigroup on the domain of the generator, corresponding to a $t^{-2}$ decay rate for the energy. Under the same localization assumption, these polynomial decay rates are optimal.
\end{abstract}
	\maketitle
	\tableofcontents
	
\section{Introduction}
\subsection{Physical background}

The Whitham--Boussinesq system belongs to a class of full-dispersion
models for the bidirectional propagation of surface water waves in an
inviscid, incompressible, and irrotational fluid of finite depth. Its
main feature is that the exact linear dispersion relation of the
water-wave problem is retained, while the nonlinear terms are described
within a weakly nonlinear Boussinesq-type approximation. We briefly
recall the physical setting leading to the model considered in this
work.

Let the fluid occupy the time-dependent domain
\[
\Omega_t
=
\left\{
(x,y)\in\mathbb R^2:
-h_0<y<h(x,t)
\right\},
\]
where \(h_0>0\) denotes the undisturbed depth and \(h=h(x,t)\) is the
free-surface elevation. Denoting the fluid velocity by
\(\mathbf u=(u,v)\), the motion is governed by the incompressible Euler
equations
\begin{equation*}
\begin{cases}
u_x+v_y=0,
\\[1mm]
\mathbf u_t+(\mathbf u\cdot\nabla)\mathbf u
=
-\dfrac{1}{\rho}\nabla p-g\mathbf e_y,
\end{cases}
\qquad
-h_0<y<h(x,t),
\end{equation*}
where \(\rho\) denotes the density of the fluid and \(g>0\) is the
gravitational acceleration. The boundary conditions at the free surface and at the flat bottom are
\begin{equation}\label{eq:physical-boundary-conditions}
\begin{cases}
h_t+uh_x-v=0,
& y=h(x,t),
\\[1mm]
p=p_{\mathrm{atm}},
& y=h(x,t),
\\[1mm]
v=0,
& y=-h_0.
\end{cases}
\end{equation}
The first condition in
\eqref{eq:physical-boundary-conditions} is the kinematic boundary
condition, expressing that the free surface moves with the fluid. The
second is the dynamic boundary condition, corresponding to continuity
of pressure across the interface, while the last condition expresses
the impermeability of the flat bottom.

Since the flow is assumed to be irrotational, there exists a velocity
potential \(\phi\) such that $u=\phi_x$ and $v=\phi_y$. The incompressibility condition then gives
\[
\Delta\phi=0
\qquad
\text{in }\Omega_t.
\]
Moreover, the Euler equations reduce to Bernoulli's equation
\[
\phi_t
+
\frac12|\nabla\phi|^2
+
gy
+
\frac{p}{\rho}
=
C(t).
\]
After absorbing \(C(t)\) into the velocity potential and using
\(p=p_{\mathrm{atm}}\) at the free surface, one obtains
\[
\phi_t
+
\frac12
\left(
\phi_x^2+\phi_y^2
\right)
+
gh
=
0
\qquad
\text{on }y=h(x,t).
\]

In the absence of external forcing and dissipative mechanisms, the
water-wave problem possesses a conservative Hamiltonian structure. In
particular, the total mechanical energy
\[
\mathcal E_{\mathrm{Euler}}(t)
=
\int_{\Omega_t}
\left[
\frac{\rho}{2}
\left(
u^2+v^2
\right)
+
\rho gy
\right]
\,dx\,dy
\]
is conserved, up to an irrelevant additive constant in the
gravitational potential energy. This conservative structure is
inherited by the linearized Whitham--Boussinesq dynamics considered
below.

To obtain a dimensionless formulation, let \(\lambda\) be a
characteristic wavelength and \(a\) a characteristic wave amplitude.
We introduce the dimensionless parameters
\[
\alpha=\frac{a}{h_0},
\qquad
\beta=\frac{h_0}{\lambda}.
\]
Here \(\alpha\) measures the strength of the nonlinearity, while
\(\beta\) measures the ratio between the fluid depth and the
characteristic wavelength.

After the scaling
\[
x=\lambda x',
\qquad
y=\lambda y',
\qquad
t=\frac{\lambda}{\sqrt{gh_0}}\,t',
\qquad
h=ah',
\]
together with the corresponding rescaling of the velocity potential,
the dimensionless water-wave problem can be written in the form
\begin{equation}\label{eq:dimensionless-water-wave}
\begin{cases}
\phi_{xx}+\phi_{yy}=0,
& -\beta<y<\alpha\beta h(x,t),
\\[1mm]
\phi_t+
\dfrac{\alpha}{2}
\left(
\phi_x^2+\phi_y^2
\right)+h=0,
& y=\alpha\beta h(x,t),
\\[1mm]
h_t+\alpha\phi_xh_x-\dfrac1\beta\phi_y=0,
& y=\alpha\beta h(x,t),
\\[1mm]
\phi_y=0,
& y=-\beta.
\end{cases}
\end{equation}

Reformulating the water-wave problem \eqref{eq:dimensionless-water-wave} on a fixed strip and expressing
the harmonic extension of the velocity potential through Fourier
multipliers naturally introduces the finite-depth factor
\(\tanh(\beta k)\). Retaining the full linear dispersion together with
the leading weakly nonlinear contributions leads to a
Whitham--Boussinesq system of the form
\begin{equation}\label{eq:nonlinear-whitham-boussinesq-introduction}
\begin{cases}
u_t
+
h_x
+
\alpha uu_x
+
\alpha\beta^2(T_\beta u)(T_\beta u)_x
=0,
\\[1mm]
h_t
-
T_\beta u
+
\alpha
\left\{
\beta^2T_\beta\bigl(hT_\beta u\bigr)+uh
\right\}_x
=0.
\end{cases}
\end{equation}
Thus, the model combines a Boussinesq-type bidirectional nonlinear
structure with the finite-depth dispersion inherited from the full
water-wave equations.

The nonlinear system
\eqref{eq:nonlinear-whitham-boussinesq-introduction} serves here only
as the physical motivation for the model. The present work is entirely
devoted to the stabilization properties of its linearization around
the zero equilibrium.

\subsection{On Whitham--Boussinesq related equations}

The full-dispersion approach underlying Whitham-type models goes back
to Whitham~\cite{Whitham1967}, who formally introduced the Whitham
equation as an alternative to the Korteweg--de Vries equation. The
basic idea is to retain the exact linear dispersion relation of the
water-wave problem while preserving a comparatively simple nonlinear
structure. In this way, the model incorporates dispersive effects
beyond those captured by the classical long-wave approximation. It
was also introduced to describe qualitative features of
large-amplitude water waves that are inaccessible to the KdV equation.

This full-dispersion philosophy was subsequently extended to
bidirectional models, leading naturally to Whitham--Boussinesq
systems. Such models provide an intermediate description between the
full water-wave equations and classical long-wave approximations:
they retain the exact, or an improved, linear dispersion relation of
the water-wave problem while approximating the nonlinear interactions
within a Boussinesq-type framework.

Several analytical aspects of Whitham--Boussinesq type systems have
been investigated in recent years. Dinvay~\cite{Dinvay} established
local well-posedness for a particular bidirectional full-dispersion
system, while Dinvay, Selberg, and Tesfahun~\cite{DinvaySelbergTesfahun}
studied low-regularity well-posedness, dispersive estimates, and global
existence for small initial data in one space dimension. A broader
framework for nonlocal quasilinear systems, including the rigorous
justification of classes of Whitham--Boussinesq systems as
approximations of the water-wave equations, was developed by
Emerald~\cite{Emerald}. Long-time well-posedness on the natural weakly
nonlinear time scale was subsequently addressed by
Paulsen~\cite{Paulsen}. More recently, Pilod, Selberg, Skoglund Taki,
and Tesfahun~\cite{PilodSelbergSkoglundTakiTesfahun} further studied
the lifespan of solutions for both the Whitham equation and a
Whitham--Boussinesq system.

Complementary questions concerning traveling waves and the qualitative
dynamics of Whitham-type models have also received considerable
attention. In particular, the existence of a highest cusped traveling
wave, conjectured by Whitham~\cite{Whitham1967}, was rigorously
established by Ehrnstr\"om and Wahl\'en~\cite{EhrnstromWahlen}.
Klein, Linares, Pilod, and Saut~\cite{KleinLinaresPilodSaut}
combined analytical arguments with extensive numerical simulations to
study the dynamics of the Whitham equation and related full-dispersion
models, including comparisons with the KdV dynamics and a discussion
of the Cauchy problem and numerical behavior of full-dispersion
Boussinesq systems. Nilsson and Wang~\cite{NilssonWang} established
the existence of solitary-wave solutions for a class of
Whitham--Boussinesq systems containing the bidirectional Whitham system
as a particular case, using a constrained variational formulation and
concentration--compactness arguments. Subsequently, Dinvay and
Nilsson~\cite{DinvayNilsson} proved the existence of solitary waves for
a particular bidirectional Whitham--Boussinesq system and obtained an
asymptotic description of these waves through variational and
concentration--compactness methods.

Spectral stability questions for Whitham--Boussinesq type models have
also been investigated. Hur and Pandey~\cite{HurPandey2019} studied a
bidirectional full-dispersion shallow-water model and analyzed the
spectral stability of small-amplitude periodic traveling waves. In
particular, they identified infinitely many collisions of purely
imaginary eigenvalues of the corresponding linearized operator and
related the onset of modulational instability to the underlying
full-dispersion structure. More recently, Creedon, Deconinck, and
Trichtchenko~\cite{CreedonDeconinckTrichtchenko2021} investigated
high-frequency spectral instabilities for small-amplitude periodic
traveling waves of a Boussinesq--Whitham system, combining a
perturbative analysis with numerical computations. Although these works
concern spectral stability about nontrivial periodic traveling waves,
rather than stabilization of the linearized dynamics around the zero
equilibrium, they illustrate the central role played by the
high-frequency spectral structure in Whitham--Boussinesq models.

The problem addressed here is nevertheless of a different nature.
While the works discussed above focus on well-posedness, traveling
waves, or the spectral stability of nonlinear coherent structures, we
study the stabilization of the \emph{linearized periodic dynamics
around the zero equilibrium}. Our main objective is to understand how
the high-frequency dispersive structure interacts with spatially
dependent damping and determines the long-time decay of the
corresponding semigroup.  A central difficulty comes from the sublinear growth of the
dispersion relation. Indeed, the conservative frequencies satisfy
\[
\omega_k\sim\beta^{-1/2}|k|^{1/2},
\qquad |k|\to\infty.
\]
As we shall see, this sublinear high-frequency behavior is directly
reflected in the resolvent of the damped generator and ultimately in
the polynomial decay rate.

Our resolvent analysis is closely related to two strands of the
stabilization literature. On the one hand, we use the high-frequency
stationary estimate of Alazard, Marzuola and Wang
\cite{AlazardMarzuolaWang} for damped fractional wave equations, in
particular for operators of the form
\[
|D|-\tau^2-i\tau\chi
\]
on the torus. This connection is natural because, after a suitable
reduction, the scalar operator arising from the
Whitham--Boussinesq system differs at high frequencies from a constant
multiple of \(|D|\) only by a smoothing Fourier multiplier.

On the other hand, polynomial stabilization in the absence of uniform
geometric damping has been extensively studied for classical damped
wave equations. In particular, Anantharaman and
L\'eautaud~\cite{AnantharamanLeautaud} investigated sharp polynomial
decay rates on the torus when the geometric control condition fails,
showing that the decay is governed by high-frequency resolvent
estimates, trapped trajectories, and the behavior of the damping
coefficient near its zero set. Although the mechanism in the present
problem is different, the obstruction here is tied to the sublinear
dispersion relation and the associated high-frequency structure of the
nonlocal system rather than to the classical geodesic flow, the same
resolvent-based strategy is relevant. In both settings, high-frequency
quasimodes provide lower bounds for the resolvent and identify the
obstruction to exponential stability, while polynomial resolvent
estimates determine the corresponding decay rates.
\subsection{Notation}

Throughout the paper, we set $\mathbb T:=\mathbb R/(2\pi\mathbb Z)$ and fix a parameter \(\beta>0\). For a periodic function
\(v:\mathbb T\to\mathbb C\), we adopt the Fourier convention
\[
\widehat v(k)
=
\frac{1}{2\pi}
\int_{\mathbb T}
v(x)e^{-ikx}\,dx,
\qquad
v(x)
=
\sum_{k\in\mathbb Z}
\widehat v(k)e^{ikx}.
\]
For \(s\in\mathbb R\), we use in $H^s(\mathbb{T})$ the Sobolev norm given by
\[
\|v\|_{H^s(\mathbb T)}^2
=
2\pi
\sum_{k\in\mathbb Z}
\langle k\rangle^{2s}
|\widehat v(k)|^2,
\qquad
\langle k\rangle=(1+k^2)^{1/2}.
\]

Now, let us introduce the nonlocal operators appearing in our work. The operator \(T_\beta\) is the periodic Fourier multiplier defined by
$$
\widehat{T_\beta v}(k)
=
-im_\beta(k)\widehat v(k),
\qquad
m_\beta(k)
=
\frac{\tanh(\beta k)}{\beta}.
$$
In particular, $\widehat{T_\beta v}(0)=0$. Since \(m_\beta\) is real, the symbol \(-im_\beta(k)\) is purely
imaginary. Hence \(T_\beta\) is skew-adjoint on \(L^2(\mathbb T)\);
in particular, for sufficiently regular functions,
\[
\langle T_\beta v,w\rangle_{L^2}
=
-\langle v,T_\beta w\rangle_{L^2}.
\]
Moreover, $m_\beta\in\ell^\infty(\mathbb Z)$,  and therefore \(T_\beta\) extends to a bounded operator on
\(H^s(\mathbb T)\) for every \(s\in\mathbb R\).

We also introduce the Fourier multiplier \(P_\beta\), defined by
$$
\widehat{P_\beta v}(k)
=
p_\beta(k)\widehat v(k),
$$
where
$$
p_\beta(k)
=
\begin{cases}
\displaystyle
\frac{\tanh(\beta k)}{\beta k},
&
k\neq0,
\\[3mm]
1,
&
k=0.
\end{cases}
$$
The value at \(k=0\) is consistent with
\[
\lim_{\xi\to0}
\frac{\tanh\xi}{\xi}
=
1.
\]
The symbol \(p_\beta\) is real, even, and strictly positive. Therefore,
\(P_\beta\) is positive and self-adjoint. More precisely,
$$
\langle P_\beta v,w\rangle_{L^2(\mathbb T)}
=
\langle v,P_\beta w\rangle_{L^2(\mathbb T)}\qquad \text{and}\qquad
\langle P_\beta v,v\rangle_{L^2(\mathbb T)}
=
2\pi\sum_{k\in\mathbb Z}
p_\beta(k)|\widehat v(k)|^2.
$$
Since
\[
p_\beta(k)
\sim
\frac{1}{\beta|k|},
\qquad
\text{as }|k|\to\infty,
\]
there exist constants \(c_\beta,C_\beta>0\) such that
$$
c_\beta\langle k\rangle^{-1}
\le
p_\beta(k)
\le
C_\beta\langle k\rangle^{-1},
\qquad
k\in\mathbb Z,
$$
where $\langle k\rangle=(1+k^2)^{1/2}$. Consequently,
\begin{equation}\label{eq:Pbeta-Hminus}
\langle P_\beta v,v\rangle_{H^{1/2},H^{-1/2}}
\asymp_\beta
\|v\|_{H^{-1/2}(\mathbb T)}^2,
\qquad
v\in H^{-1/2}(\mathbb T).
\end{equation}
Thus, \(P_\beta\) is a Fourier multiplier of order \(-1\) and defines
an isomorphism
\[
P_\beta:
H^s(\mathbb T)
\longrightarrow
H^{s+1}(\mathbb T),
\]
for every \(s\in\mathbb R\). In particular, \(P_\beta^{1/2}\) is the positive self-adjoint Fourier
multiplier with symbol \(p_\beta(k)^{1/2}\), and
\[
P_\beta^{1/2}:
H^{-1/2}(\mathbb T)
\longrightarrow
L^2(\mathbb T)
\]
is an isomorphism. Moreover, $\|P_\beta^{1/2}v\|_{L^2(\mathbb T)}^2=
\langle P_\beta v,v\rangle_{H^{1/2},H^{-1/2}}$.  
The fundamental identity relates the two nonlocal operators
\begin{equation}\label{xp-beta}
\partial_xP_\beta=-T_\beta.
\end{equation}
Equivalently, $kp_\beta(k)=m_\beta(k)$, for $k\in\mathbb Z$. Furthermore, we know that \(P_\beta\) and \(\partial_x\) are Fourier multipliers, so they commute: $P_\beta\partial_x=\partial_xP_\beta$.

\subsection{Formulation of the problem and main result}
The present work is concerned exclusively with the linearized
Whitham--Boussinesq dynamics. Neglecting the nonlinear terms in
\eqref{eq:nonlinear-whitham-boussinesq-introduction}, we obtain
\begin{equation}\label{eq:linear-whitham-boussinesq-introduction}
\begin{cases}
h_t-T_\beta u=0,
\\[1mm]
u_t+h_x=0.
\end{cases}
\end{equation}

An equivalent formulation is useful for relating
\eqref{eq:linear-whitham-boussinesq-introduction} to other
Whitham--Boussinesq models considered in the literature. Indeed, using
the identity \eqref{xp-beta},  system \eqref{eq:linear-whitham-boussinesq-introduction} can be written
as
\[
\begin{cases}
h_t+P_\beta u_x=0,\\
u_t+h_x=0.
\end{cases}
\]
Since
\[
P_\beta
=
\frac{\tanh(\beta|D|)}{\beta|D|},
\]
the parameter identification $\beta=\sqrt{\epsilon}$ yields
\[
P_\beta=\mathcal L_\epsilon^2,
\qquad
\mathcal L_\epsilon
=
\left(
\frac{\tanh(\sqrt{\epsilon}|D|)}
{\sqrt{\epsilon}|D|}
\right)^{1/2}.
\]
Consequently, after identifying \(h=\eta\), the above system coincides
with the linearization around the zero equilibrium of the
one-dimensional full-dispersion Boussinesq system studied, among
others, in
\cite{PilodSelbergSkoglundTakiTesfahun,KleinLinaresPilodSaut}; see also
the references therein. More precisely, consider
\[
\begin{cases}
\eta_t+\mathcal L_\epsilon^2u_x
+\epsilon(\eta u)_x=0,\\
u_t+\eta_x+\epsilon uu_x=0.
\end{cases}
\]
Indeed, neglecting the quadratic terms gives
\[
\begin{cases}
\eta_t+\mathcal L_\epsilon^2u_x=0,\\
u_t+\eta_x=0.
\end{cases}
\]
Thus, although the nonlinear models are formulated differently, their
linearized dynamics agree after the above identification of the
dispersive parameters.

Throughout the paper, we retain formulation
\eqref{eq:linear-whitham-boussinesq-introduction}, which is naturally
adapted to the operators $T_\beta$ and $P_\beta$ and to the energy
framework developed below. Eliminating \(u\) yields the nonlocal wave
equation\footnote{Nonlocal scalar formulations of this type also arise directly in the
analysis of finite-depth surface waves; see, for instance,
Artiles, Kraenkel, and Manna~\cite{ArtilesKraenkelManna2009}.}
\[
h_{tt}+T_\beta h_x=0.
\]
In Fourier variables, each nonzero spatial mode satisfies
\[
\partial_t^2\widehat h(k,t)
+
\omega_k^2\widehat h(k,t)
=
0,
\]
where
\[
\omega_k
=
\sqrt{
\frac{k\tanh(\beta k)}{\beta}
}.
\]
Thus, each Fourier mode evolves as a conservative oscillator whose frequency is determined by the finite-depth dispersion relation.

The natural energy associated with
\eqref{eq:linear-whitham-boussinesq-introduction} is
\begin{equation}\label{eq:linear-energy-introduction}
E(t)
=
\frac12
\|h(t)\|_{L^2(\mathbb T)}^2
+
\frac12
\langle
P_\beta u(t),u(t)
\rangle_{H^{1/2},H^{-1/2}}.
\end{equation}
For sufficiently regular solutions,
\[
\frac{d}{dt}E(t)=0.
\]

This naturally leads to the phase space $\mathcal H=L^2(\mathbb T)\times H^{-1/2}(\mathbb T)$,
endowed with the energy inner product
\[
\left\langle
(h,u),(\varphi,\psi)
\right\rangle_{\mathcal H}
=
\langle h,\varphi\rangle_{L^2}
+
\langle P_\beta u,\psi\rangle_{H^{1/2},H^{-1/2}}.
\]
With this choice,
\[
E(t)
=
\frac12
\|(h(t),u(t))\|_{\mathcal H}^2.
\]
Moreover, by \eqref{eq:Pbeta-Hminus}, the energy norm is equivalent to
the standard product norm on
\(L^2(\mathbb T)\times H^{-1/2}(\mathbb T)\).

The conservative character of
\eqref{eq:linear-whitham-boussinesq-introduction} motivates the introduction of spatially dependent damping. To describe the feedback mechanism,
consider first the controlled system
\begin{equation}\label{eq:controlled-system-torus}
\begin{cases}
h_t-T_\beta u=a(x)f,
    & (x,t)\in\mathbb T\times(0,T),
\\[1mm]
u_t+h_x=b(x)g,
    & (x,t)\in\mathbb T\times(0,T),
\\[1mm]
h(x,0)=h_0(x),\qquad u(x,0)=u_0(x),
    & x\in\mathbb T,
\end{cases}
\end{equation}
where
\begin{equation}\label{a-b-ab}
a,b\in C^\infty(\mathbb T;\mathbb R),
\qquad
a,b\ge0,
\qquad
a\not\equiv0,
\qquad
b\not\equiv0.
\end{equation}
Remember that \(a\) and \(b\) are continuous and nontrivial, so one may choose
nonempty open sets
\[
\omega_1\Subset\{x\in\mathbb T:a(x)>0\},
\qquad
\omega_2\Subset\{x\in\mathbb T:b(x)>0\},
\]
and constants \(a_0,b_0>0\) such that
\[
a(x)\ge a_0
\quad\text{on }\omega_1,
\qquad
b(x)\ge b_0
\quad\text{on }\omega_2.
\]

We choose the dissipative feedback laws $f=-ah$ and $g=-bP_\beta u$. The corresponding closed-loop system associated with \eqref{eq:controlled-system-torus} is therefore
$$
\begin{cases}
h_t-T_\beta u+a^2h=0,    & (x,t)\in\mathbb T\times(0,T),
\\[1mm]
u_t+h_x+b^2P_\beta u=0,    & (x,t)\in\mathbb T\times(0,T),\\[1mm]
h(x,0)=h_0(x),\qquad u(x,0)=u_0(x),
    & x\in\mathbb T.
\end{cases}
$$
Writing
\[
Z=
\begin{pmatrix}
h\\
u
\end{pmatrix},
\qquad
Z_0=
\begin{pmatrix}
h_0\\
u_0
\end{pmatrix},
\]
the closed-loop system takes the abstract form
\begin{equation}\label{aad-z}
\begin{cases}
Z_t=\mathcal A_d Z,\\
Z(0)=Z_0,
\end{cases}
\qquad
\mathcal A_d=\mathcal A-\mathcal B\mathcal B^*.
\end{equation}
With this notation, the energy defined in
\eqref{eq:linear-energy-introduction} satisfies
\[
E(t)
=
\frac12
\|Z(t)\|_{\mathcal H}^2.
\]
Here, $\mathcal B(f,g)=(af,bg)^\top$ and $\mathcal B^*(h,u)=(ah,bP_\beta u)^\top$, denote the control operator and its adjoint, respectively, while $\mathcal D=\mathcal B\mathcal B^*$ is the damping operator; we refer the reader to Section~\ref{sec:well-posedness} for more details. 

For sufficiently regular solutions, the feedback has the
expected dissipative effect:
\[
\frac{d}{dt}E(t)
=
-\|ah(t)\|_{L^2(\mathbb T)}^2
-
\|bP_\beta u(t)\|_{L^2(\mathbb T)}^2
\le0.
\]
The main problem addressed in this paper is therefore to determine
whether this dissipation forces
\[
E(t)\longrightarrow0,
\qquad
\text{as }t\to\infty,
\]
and, more importantly, to determine the optimal rate at which this
convergence occurs.

The main purpose of this work is to determine the long-time behavior of
the damped linearized Whitham--Boussinesq system and, in particular, to
identify the optimal decay rate of its energy. Our main result shows
that the decay is polynomial and that, under a genuine localization
assumption on the damping, the corresponding polynomial rate is
optimal.

\begin{theorem}[Optimal polynomial stabilization]
\label{thm:main-polynomial-stabilization}
Let \(S_d(t)\) be the contraction semigroup generated by
\(\mathcal A_d\) on $\mathcal H$. Assume that the localization functions $a,b$ satisfy \eqref{a-b-ab}. Then
\[
i\mathbb R\subset\rho(\mathcal A_d),
\]
and there exists a constant \(C>0\) such that
\begin{equation}\label{eq:semigroup-polynomial-decay}
\left\|
S_d(t)\mathcal A_d^{-1}
\right\|_{\mathcal L(\mathcal H)}
\le
\frac{C}{t},
\qquad
t\ge1.
\end{equation}
Consequently, for every $Z_0\in D(\mathcal A_d),$
the corresponding solution $Z(t)=S_d(t)Z_0$ of \eqref{aad-z}
satisfies
\begin{equation}\label{eq:state-polynomial-decay}
\|Z(t)\|_{\mathcal H}
\le
\frac{C}{t}
\|\mathcal A_d Z_0\|_{\mathcal H},
\qquad
t\ge1.
\end{equation}
In particular,
\begin{equation}\label{eq:state-polynomial-domain}
\|Z(t)\|_{\mathcal H}
\le
\frac{C}{t}
\|Z_0\|_{D(\mathcal A_d)},
\qquad
t\ge1.
\end{equation}
Consequently, the energy satisfies
\begin{equation}\label{eq:energy-polynomial-decay}
E(t)
\le
\frac{C}{t^2}
\|Z_0\|_{D(\mathcal A_d)}^2,
\qquad
t\ge1.
\end{equation}
If, in addition, the damping acting on the first component is genuinely
localized, in the sense that there exists a nonempty open set
\[
\mathcal O\Subset
\mathbb T\setminus\operatorname{supp}(a),
\]
then the rate in \eqref{eq:semigroup-polynomial-decay} is optimal in
the polynomial scale. Consequently, the corresponding \(t^{-2}\)
energy decay rate is also optimal in the polynomial scale.
\end{theorem}

\subsection{Proof strategy and organization of the paper}

We now briefly describe the main ideas behind the proof of Theorem \ref{thm:main-polynomial-stabilization}.
\vspace{0.1cm}

\noindent $\bullet$ \textbf{The first step is to establish the functional and spectral framework.}
The conservative operator
\[
\mathcal A
\binom{h}{u}
=
\binom{T_\beta u}{-h_x},
\qquad
D(\mathcal A)
=
H^{1/2}(\mathbb T)\times L^2(\mathbb T),
\]
is skew-adjoint in the natural energy space, while the damped operator
\(\mathcal A_d\) is maximal dissipative. Hence \(\mathcal A\) generates
a unitary \(C_0\)-group and \(\mathcal A_d\) generates the contraction
semigroup \(S_d(t)\). Moreover, $i\mathbb R\subset\rho(\mathcal A_d)$, so the Arendt--Batty-Lyubich--V\~u theorem \cite{ArendtBatty,LyubichVu}, yields strong stability.

\vspace{0.1cm}

\noindent $\bullet$ \textbf{The second step is to identify the high-frequency spectral structure.}
The conservative frequencies satisfy
\[
\omega_k
=
\sqrt{\frac{k\tanh(\beta k)}{\beta}}
\sim
\frac{1}{\sqrt{\beta}}|k|^{1/2},
\qquad
|k|\to\infty.
\]
In particular, the corresponding family of frequencies does not satisfy
the uniform gap condition underlying classical Ingham-type approaches
used for periodic dispersive equations, as in
Russell and Zhang~\cite{RussellZhang93} and, for coupled systems,
Capistrano--Filho, Komornik and Pazoto~\cite{CapKoPaz20}.
Thus, the classical nonharmonic Fourier approach cannot be applied
directly in the present setting, and the quantitative stabilization
problem is instead treated through high-frequency resolvent estimates.

\vspace{0.1cm}

\noindent $\bullet$ \textbf{The third step is the resolvent analysis.}
We prove the upper bound
\[
\left\|
(i\tau I-\mathcal A_d)^{-1}
\right\|_{\mathcal L(\mathcal H)}
\lesssim
|\tau|,
\qquad
|\tau|\to\infty.
\]
The key reduction is obtained through the change of variable
$v=P_\beta^{1/2}u$, which leads to the scalar operator
$L_\beta=T_\beta\partial_x$ satisfying
\[
\beta L_\beta=|D|+R_\beta,
\]
where \(R_\beta\) is a smoothing Fourier multiplier. This representation
allows us to adapt the high-frequency stationary resolvent estimate of
Alazard, Marzuola, and Wang~\cite{AlazardMarzuolaWang} for the
fractional-wave operator to the present setting, treating \(R_\beta\)
through a compactness argument.

On the other hand, under the localization assumption $\mathcal O\Subset
\mathbb T\setminus\operatorname{supp}(a),$ 
high-frequency quasimodes concentrated in an undamped region give
\[
\left\|
(i\omega_NI-\mathcal A_d)^{-1}
\right\|_{\mathcal L(\mathcal H)}
\gtrsim
\omega_N
\]
along a sequence \(\omega_N\to\infty\). Hence the exponent \(1\) in
the polynomial resolvent upper bound is sharp.

\vspace{0.1cm}

\noindent $\bullet$ \textbf{The fourth step is the application of the
Borichev--Tomilov theorem.}  Finally, the Borichev--Tomilov theorem \cite{BorichevTomilov2010} converts the resolvent upper
bound into
\[
\left\|
S_d(t)\mathcal A_d^{-1}
\right\|_{\mathcal L(\mathcal H)}
=
O(t^{-1}),
\qquad
t\to\infty.
\]
Since
\[
E(t)=\frac12\|Z(t)\|_{\mathcal H}^2,
\]
this yields
\[
E(t)=O(t^{-2}),
\]
for initial data in \(D(\mathcal A_d)\). Under genuinely localized
damping, the quasimode lower bound shows that no faster polynomial
decay rate can hold.

\vspace{0.2cm}
The remainder of the paper is organized as follows. In
Section~\ref{sec:well-posedness}, we establish the functional framework
and the well-posedness of the conservative and damped linearized
Whitham--Boussinesq systems. Section~\ref{sec:spectral-analysis} is
devoted to the spectral analysis of the conservative generator, with
particular emphasis on the high-frequency behavior of the dispersion
relation. In Section~\ref{sec:spectral-damped}, we study the damped
generator and prove strong stability of the associated semigroup. Section~\ref{sec:high-frequency-resolvent} contains the high-frequency
resolvent analysis, including the construction of quasimodes and the
optimal linear resolvent bound. In
Section~\ref{sec:polynomial-stabilization}, these estimates are combined
with the Borichev--Tomilov theorem to obtain the optimal polynomial
stabilization result, that is, Theorem \ref{thm:main-polynomial-stabilization}. Finally,
Section~\ref{sec:final-considerations} contains concluding remarks and
possible directions for further investigation.

\section{Well-posedness}
\label{sec:well-posedness}

In this section, we establish the well-posedness of the conservative and
damped systems. We first analyze the conservative dynamics and prove
that the corresponding operator is skew-adjoint. The damped system is
then treated as a bounded dissipative perturbation of the conservative
operator.

\subsection{The energy space}

Recall that the natural energy space is $\mathcal H
=
L^2(\mathbb T)\times H^{-1/2}(\mathbb T).
$ For $
Z=(h,u)^\top, \ W=(\varphi,\psi)^\top
\in\mathcal H$,
we define
\[
\langle Z,W\rangle_{\mathcal H}
=
\langle h,\varphi\rangle_{L^2(\mathbb T)}
+
\langle P_\beta u,\psi\rangle_{H^{1/2},H^{-1/2}}.
\]
Equivalently, using the Fourier convention adopted throughout the
paper,
\[
\langle Z,W\rangle_{\mathcal H}
=
2\pi
\sum_{k\in\mathbb Z}
\widehat h(k)\overline{\widehat\varphi(k)}
+
2\pi
\sum_{k\in\mathbb Z}
p_\beta(k)\widehat u(k)\overline{\widehat\psi(k)}.
\]
The second series is absolutely convergent by the Cauchy--Schwarz
inequality and the equivalence
\[
p_\beta(k)\asymp_\beta\langle k\rangle^{-1}.
\]
Observe that this defines an inner
product on \(\mathcal H\) once we have that  \(P_\beta\) is positive and self-adjoint. Moreover, by
\eqref{eq:Pbeta-Hminus}, the induced norm satisfies
\[
\|Z\|_{\mathcal H}^2
\asymp_\beta
\|h\|_{L^2(\mathbb T)}^2
+
\|u\|_{H^{-1/2}(\mathbb T)}^2.
\]

Consider the conservative system
\begin{equation*}\label{conservative-system}
\begin{cases}
h_t-T_\beta u=0,
\\[1mm]
u_t+h_x=0,\\[1mm]
h(x,0)=h_0(x),\qquad u(x,0)=u_0(x).\end{cases}
\end{equation*}
We define the operator
\[
\mathcal A
\binom{h}{u}
=
\binom{T_\beta u}{-h_x},
\]
with domain
\begin{equation}\label{domainA}
D(\mathcal A)
=
H^{1/2}(\mathbb T)
\times
L^2(\mathbb T).
\end{equation}
Since \(H^{1/2}(\mathbb T)\) is dense in \(L^2(\mathbb T)\) and
\(L^2(\mathbb T)\) is dense in \(H^{-1/2}(\mathbb T)\), the domain
\(D(\mathcal A)\) is dense in \(\mathcal H\). Moreover, the symbol of \(T_\beta\) is bounded, thus
\[
T_\beta:
L^2(\mathbb T)
\longrightarrow
L^2(\mathbb T)
\]
is continuous, while
\[
\partial_x:
H^{1/2}(\mathbb T)
\longrightarrow
H^{-1/2}(\mathbb T)
\]
is also continuous. Hence
\[
\mathcal A:
D(\mathcal A)
\subset
\mathcal H
\longrightarrow
\mathcal H
\]
is well defined. With \(Z=(h,u)^\top\), system
\eqref{conservative-system} can therefore be written as
\begin{equation}\label{conservative-system-abs}
\begin{cases}
Z_t=\mathcal A Z,\\
Z(0)=Z_0.
\end{cases}
\end{equation}

We next prove that the conservative operator is skew-symmetric with
respect to the energy inner product.

\begin{proposition}\label{prop:skew-symmetry}
For every \(Z,W\in D(\mathcal A)\),
\[
\langle\mathcal AZ,W\rangle_{\mathcal H}
=
-
\langle Z,\mathcal AW\rangle_{\mathcal H}.
\]
Consequently,
\[
\operatorname{Re}
\langle\mathcal AZ,Z\rangle_{\mathcal H}
=
0,
\qquad
Z\in D(\mathcal A).
\]
\end{proposition}

\begin{proof}
Let
$
Z=(h,u)^\top,
\
W=(\varphi,\psi)^\top
\in D(\mathcal A).
$ Using Parseval's identity and the definition of the energy inner
product, we obtain
\begin{align*}
\langle\mathcal AZ,W\rangle_{\mathcal H}
&=
\langle T_\beta u,\varphi\rangle_{L^2(\mathbb T)}
+
\langle P_\beta(-h_x),\psi\rangle_{H^{1/2},H^{-1/2}}
\\
&=
-2\pi
\sum_{k\in\mathbb Z}
im_\beta(k)
\widehat u(k)
\overline{\widehat\varphi(k)}
-
2\pi
\sum_{k\in\mathbb Z}
ikp_\beta(k)
\widehat h(k)
\overline{\widehat\psi(k)}.
\end{align*}
As  $kp_\beta(k)=m_\beta(k)$, with  $k\in\mathbb Z$, we obtain
\[
\langle\mathcal AZ,W\rangle_{\mathcal H}
=
-2\pi
\sum_{k\in\mathbb Z}
im_\beta(k)
\widehat u(k)
\overline{\widehat\varphi(k)}
-
2\pi
\sum_{k\in\mathbb Z}
im_\beta(k)
\widehat h(k)
\overline{\widehat\psi(k)}.
\]

Similarly,
\begin{align*}
\langle Z,\mathcal AW\rangle_{\mathcal H}
&=
\langle h,T_\beta\psi\rangle_{L^2(\mathbb T)}
+
\langle P_\beta u,-\varphi_x\rangle_{H^{1/2},H^{-1/2}}
\\
&=
2\pi
\sum_{k\in\mathbb Z}
im_\beta(k)
\widehat h(k)
\overline{\widehat\psi(k)}
+
2\pi
\sum_{k\in\mathbb Z}
ikp_\beta(k)
\widehat u(k)
\overline{\widehat\varphi(k)}
\\
&=
2\pi
\sum_{k\in\mathbb Z}
im_\beta(k)
\widehat h(k)
\overline{\widehat\psi(k)}
+
2\pi
\sum_{k\in\mathbb Z}
im_\beta(k)
\widehat u(k)
\overline{\widehat\varphi(k)}.
\end{align*}
Therefore,
\[
\langle\mathcal AZ,W\rangle_{\mathcal H}
=
-
\langle Z,\mathcal AW\rangle_{\mathcal H}.
\]
Note that all the preceding series are absolutely convergent. Indeed,
\(m_\beta\in\ell^\infty(\mathbb Z)\), while
\[
u,\psi\in L^2(\mathbb T),
\qquad
h,\varphi\in H^{1/2}(\mathbb T)
\hookrightarrow L^2(\mathbb T),
\]
so the conclusion follows from the Cauchy--Schwarz inequality.

Finally, taking \(W=Z\), we obtain
\[
\langle\mathcal AZ,Z\rangle_{\mathcal H}
=
-
\overline{
\langle\mathcal AZ,Z\rangle_{\mathcal H}
},
\]
and hence $\operatorname{Re}
\langle\mathcal AZ,Z\rangle_{\mathcal H}
=
0$.
This completes the proof.
\end{proof}

The skew-symmetry established above does not by itself imply that
\(\mathcal A\) is skew-adjoint. To obtain maximality, we next study the
operators \(\lambda I\pm\mathcal A\) in Fourier variables.

\subsection{The resolvent of the conservative operator}

We now study the resolvent of the conservative operator. Owing to the
translation invariance of the system, the resolvent equation can be
solved explicitly in Fourier variables.

Let $F=(f,g)^\top\in\mathcal H$ and \(\lambda>0\). We consider
\begin{equation}\label{resolvent}
\begin{cases}
(\lambda I-\mathcal A)U=F,
\\
U=(h,u)^\top.
\end{cases}
\end{equation}
Taking Fourier coefficients in \eqref{resolvent}, we obtain
\[
\begin{cases}
\lambda\widehat h(k)
+
im_\beta(k)\widehat u(k)
=
\widehat f(k),
\\[1mm]
ik\widehat h(k)
+
\lambda\widehat u(k)
=
\widehat g(k),
\end{cases}
\]
or, equivalently,
\[
M_\lambda(k)
\binom{\widehat h(k)}
{\widehat u(k)}
=
\binom{\widehat f(k)}
{\widehat g(k)},
\]
where
\begin{equation*}
M_\lambda(k)
=
\begin{pmatrix}
\lambda
&
im_\beta(k)
\\
ik
&
\lambda
\end{pmatrix}.
\end{equation*}
We first record the invertibility of the Fourier matrices.

\begin{lemma}\label{lemma:invertibility}
For every \(\lambda>0\) and every \(k\in\mathbb Z\), the matrix
\(M_\lambda(k)\) is invertible. Moreover,
\[
\det M_\lambda(k)
=
\lambda^2+km_\beta(k).
\]
\end{lemma}

\begin{proof}
A direct computation gives
\[
\det M_\lambda(k)
=
\lambda^2-(im_\beta(k))(ik)
=
\lambda^2+km_\beta(k).
\]
Since
\[
km_\beta(k)
=
\frac{k\tanh(\beta k)}{\beta}
\ge0,
\]
we have
\[
\det M_\lambda(k)
\ge
\lambda^2
>
0.
\]
Thus \(M_\lambda(k)\) is invertible for every \(k\in\mathbb Z\).
\end{proof}

The inverse matrix is therefore given by
\[
M_\lambda(k)^{-1}
=
\frac1{\lambda^2+km_\beta(k)}
\begin{pmatrix}
\lambda
&
-im_\beta(k)
\\
-ik
&
\lambda
\end{pmatrix}.
\]
Hence
\begin{equation}\label{hkformula}
\widehat h(k)
=
\frac{
\lambda\widehat f(k)
-
im_\beta(k)\widehat g(k)}
{\lambda^2+km_\beta(k)},
\end{equation}
and
\begin{equation}\label{ukformula}
\widehat u(k)
=
\frac{
-ik\widehat f(k)
+
\lambda\widehat g(k)}
{\lambda^2+km_\beta(k)}.
\end{equation}

We can now establish the surjectivity of
\(\lambda I-\mathcal A\).

\begin{proposition}\label{prop:resolvent}
For every \(\lambda>0\),
\[
\lambda I-\mathcal A:
D(\mathcal A)
\longrightarrow
\mathcal H
\]
is an isomorphism. Moreover,
\[
(\lambda I-\mathcal A)^{-1}
\in
\mathcal L(\mathcal H,D(\mathcal A)),
\]
where \(D(\mathcal A)\) is given by \eqref{domainA}.
\end{proposition}

\begin{proof}
Let \(F=(f,g)^\top\in\mathcal H\). By
Lemma~\ref{lemma:invertibility}, formulas
\eqref{hkformula}--\eqref{ukformula} uniquely determine the Fourier
coefficients of \(U=(h,u)^\top\). It remains to prove that
\[
U\in
D(\mathcal A)
=
H^{1/2}(\mathbb T)\times L^2(\mathbb T).
\]

Observe that
\[
km_\beta(k)
=
\frac{k\tanh(\beta k)}{\beta}
=
\frac{|k|\tanh(\beta|k|)}{\beta}
\asymp_\beta
|k|.
\]
Consequently,
\begin{equation}\label{rk}
\lambda^2+km_\beta(k)
\asymp_{\lambda,\beta}
1+|k|,
\qquad
k\in\mathbb Z.
\end{equation}

We first estimate \(h\). Using \eqref{hkformula}, the boundedness of
\(m_\beta\), and \eqref{rk}, we obtain
\begin{equation}\label{hk2}
(1+|k|)|\widehat h(k)|^2
\lesssim_{\lambda,\beta}
|\widehat f(k)|^2
+
\frac1{1+|k|}
|\widehat g(k)|^2.
\end{equation}
Thanks to the fact that $f\in L^2(\mathbb T)$ and $g\in H^{-1/2}(\mathbb T)$, summing over \(k\in\mathbb Z\) yields
\[
\sum_{k\in\mathbb Z}
(1+|k|)
|\widehat h(k)|^2
<
\infty.
\]
Hence $h\in H^{1/2}(\mathbb T)$.

Next, to estimate \(u\), using \eqref{ukformula} and \eqref{rk}, we obtain
\begin{equation}\label{uk2}
|\widehat u(k)|^2
\lesssim_{\lambda,\beta}
|\widehat f(k)|^2
+
\frac1{1+|k|}
|\widehat g(k)|^2.
\end{equation}
Therefore,
\[
\sum_{k\in\mathbb Z}
|\widehat u(k)|^2
<
\infty,
\]
and consequently $u\in L^2(\mathbb T)$. 
Thus
\[
U=(h,u)^\top
\in
H^{1/2}(\mathbb T)\times L^2(\mathbb T)
=
D(\mathcal A),
\]
which proves the surjectivity of
\[
\lambda I-\mathcal A:
D(\mathcal A)\longrightarrow\mathcal H.
\]
Injectivity follows from Lemma~\ref{lemma:invertibility}, since the
Fourier matrices \(M_\lambda(k)\) are invertible for every
\(k\in\mathbb Z\). Hence \(\lambda I-\mathcal A\) is an isomorphism.

The preceding estimates \eqref{hk2} and \eqref{uk2} also yield
\[
\|h\|_{H^{1/2}(\mathbb T)}^2
+
\|u\|_{L^2(\mathbb T)}^2
\le
C_{\lambda,\beta}
\left(
\|f\|_{L^2(\mathbb T)}^2
+
\|g\|_{H^{-1/2}(\mathbb T)}^2
\right).
\]
Since \(\mathcal A\) is continuous from
\(H^{1/2}(\mathbb T)\times L^2(\mathbb T)\) into \(\mathcal H\), this
also controls the graph norm on \(D(\mathcal A)\). Therefore,
\[
(\lambda I-\mathcal A)^{-1}
\in
\mathcal L(\mathcal H,D(\mathcal A)).
\]
This completes the proof.
\end{proof}

In the next subsection, we establish the analogous property for
\(\lambda I+\mathcal A\). Together with
Proposition~\ref{prop:skew-symmetry}, this will imply that
\(\mathcal A\) is skew-adjoint and hence, by Stone's theorem, generates
a unitary \(C_0\)-group on \(\mathcal H\).

\subsection{Skew-adjointness and generation of the conservative group}

We have established that \(\lambda I-\mathcal A\) is an isomorphism for
every \(\lambda>0\). We now prove the analogous property for
\(\lambda I+\mathcal A\).

\begin{proposition}\label{prop:range-plus-A}
For every \(\lambda>0\),
\[
\lambda I+\mathcal A:
D(\mathcal A)
\longrightarrow
\mathcal H
\]
is an isomorphism. Moreover,
\[
(\lambda I+\mathcal A)^{-1}
\in
\mathcal L(\mathcal H,D(\mathcal A)).
\]
\end{proposition}

\begin{proof}
Let \(F=(f,g)^\top\in\mathcal H\). We consider
\begin{equation}\label{eq:resolvent-plus}
(\lambda I+\mathcal A)U=F,
\qquad
U=(h,u)^\top.
\end{equation}
In Fourier variables, \eqref{eq:resolvent-plus} becomes
\[
\begin{cases}
\lambda\widehat h(k)
-
im_\beta(k)\widehat u(k)
=
\widehat f(k),
\\[1mm]
-ik\widehat h(k)
+
\lambda\widehat u(k)
=
\widehat g(k).
\end{cases}
\]
Equivalently,
\[
N_\lambda(k)
\binom{\widehat h(k)}
{\widehat u(k)}
=
\binom{\widehat f(k)}
{\widehat g(k)},
\]
where
\[
N_\lambda(k)
=
\begin{pmatrix}
\lambda & -im_\beta(k)
\\
-ik & \lambda
\end{pmatrix}.
\]
Note that
\[
\det N_\lambda(k)
=
\lambda^2+km_\beta(k)
\ge
\lambda^2>0,
\]
so the matrix \(N_\lambda(k)\) is invertible for every \(k\in\mathbb Z\),
with
\[
N_\lambda(k)^{-1}
=
\frac1{\lambda^2+km_\beta(k)}
\begin{pmatrix}
\lambda & im_\beta(k)
\\
ik & \lambda
\end{pmatrix}.
\]
Consequently,
\[
\widehat h(k)
=
\frac{
\lambda\widehat f(k)
+
im_\beta(k)\widehat g(k)}
{\lambda^2+km_\beta(k)},
\]
and
\[
\widehat u(k)
=
\frac{
ik\widehat f(k)
+
\lambda\widehat g(k)}
{\lambda^2+km_\beta(k)}.
\]

These formulas differ from those obtained for
\((\lambda I-\mathcal A)^{-1}\) only in the signs of the off-diagonal
terms. Therefore, the estimates established in
Proposition~\ref{prop:resolvent} apply without modification. In
particular,
\[
\|h\|_{H^{1/2}(\mathbb T)}^2
+
\|u\|_{L^2(\mathbb T)}^2
\le
C_{\lambda,\beta}
\left(
\|f\|_{L^2(\mathbb T)}^2
+
\|g\|_{H^{-1/2}(\mathbb T)}^2
\right).
\]
Hence \(U\in D(\mathcal A)\), which proves the surjectivity of
\(\lambda I+\mathcal A\). Injectivity follows from the invertibility of
\(N_\lambda(k)\) for every \(k\in\mathbb Z\). Thus
\(\lambda I+\mathcal A\) is an isomorphism, and the preceding estimate
gives
\[
(\lambda I+\mathcal A)^{-1}
\in
\mathcal L(\mathcal H,D(\mathcal A)),
\]
giving the proof.
\end{proof}

We can now identify the conservative generator.

\begin{theorem}\label{thm:skew-adjoint-A}
The operator $\mathcal A:
D(\mathcal A)\subset\mathcal H
\longrightarrow
\mathcal H$
is skew-adjoint, that is,  $\mathcal A^*=-\mathcal A$.
\end{theorem}

\begin{proof}
By Proposition~\ref{prop:skew-symmetry}, \(\mathcal A\) is densely
defined and skew-symmetric. Moreover,
Propositions~\ref{prop:resolvent} and~\ref{prop:range-plus-A} yield
$\operatorname{Ran}(\lambda I-\mathcal A)
=
\operatorname{Ran}(\lambda I+\mathcal A)
=
\mathcal H,$
for every \(\lambda>0\). The standard maximality criterion for
skew-symmetric operators therefore implies that $\mathcal A^*=-\mathcal A$.
\end{proof}

As a consequence of Stone's theorem, we obtain the conservative
well-posedness result.

\begin{theorem}\label{thm:well-posedness-conservative}
The operator \(\mathcal A\) generates a unitary \(C_0\)-group $\{S(t)\}_{t\in\mathbb R}=
\{e^{t\mathcal A}\}_{t\in\mathbb R}$ on \(\mathcal H\). Consequently, for every
\(Z_0=(h_0,u_0)^\top\in\mathcal H\), system
\eqref{conservative-system-abs} admits a unique mild solution $Z(t)=S(t)Z_0 \in C(\mathbb R;\mathcal H)$, and
\[
\|Z(t)\|_{\mathcal H}
=
\|Z_0\|_{\mathcal H},
\qquad
t\in\mathbb R.
\]
If, in addition, \(Z_0\in D(\mathcal A)\), then
\[
Z
\in
C(\mathbb R;D(\mathcal A))
\cap
C^1(\mathbb R;\mathcal H),
\]
and \(Z\) is the unique classical solution of \eqref{conservative-system-abs}.
\end{theorem}

\begin{proof} By the previous result \(\mathcal A\) is skew-adjoint; thus Stone's theorem
(see, for instance, \cite{Pazy}) implies that \(\mathcal A\) generates
a strongly continuous unitary group on \(\mathcal H\). The existence,
uniqueness, and regularity statements follow from the standard theory
of \(C_0\)-groups, while the conservation of the \(\mathcal H\)-norm
follows from unitarity.
\end{proof}

\subsection{The damped system}\label{sec-damped}
Let $\mathcal U
=
L^2(\mathbb T)\times L^2(\mathbb T)$
be the control space, and define
\[
\mathcal B:
\mathcal U
\longrightarrow
\mathcal H,
\qquad
\mathcal B(f,g)
=
(af,bg)^\top.
\]
Since \(a,b\in C^\infty(\mathbb T)\), the operator
\(\mathcal B\) belongs to
\(\mathcal L(\mathcal U,\mathcal H)\). 
The expression for the adjoint $\mathcal B^*$ follows directly from
the energy inner product on $\mathcal H$. Indeed, for
$(f,g)\in\mathcal U$ and $(h,u)\in\mathcal H$, we have
\[
\begin{aligned}
\left\langle
\mathcal B(f,g),(h,u)
\right\rangle_{\mathcal H}
=
\langle af,h\rangle_{L^2}
+
\langle P_\beta(bg),u\rangle
=
\langle f,ah\rangle_{L^2}
+
\langle g,bP_\beta u\rangle_{L^2},
\end{aligned}
\]
where we have used that $P_\beta$ is self-adjoint and that $a$ and
$b$ are real-valued. Hence
\[
\mathcal B^*(h,u)
=
\left(
ah,
bP_\beta u
\right)^\top.
\]
Consequently, setting $\mathcal D:=\mathcal B\mathcal B^*$, we obtain
\[
\mathcal D(h,u)
=
\left(
a^2h,
b^2P_\beta u
\right)^\top.
\]
Therefore, the closed-loop generator takes the form
\[
\mathcal A_d
=
\mathcal A-\mathcal D
=
\mathcal A-\mathcal B\mathcal B^*.
\]
Summarizing, the closed-loop system can therefore be written as
\begin{equation}\label{dp-s}
\begin{cases}
Z_t=\mathcal A_d Z,
\\[1mm]
Z(0)=Z_0,
\end{cases}
\qquad
\mathcal A_d
:=
\mathcal A-\mathcal D,
\end{equation}
with
$
D(\mathcal A_d)
=
D(\mathcal A).
$
More explicitly,
\[
\mathcal A_d(h,u)
=
\left(
T_\beta u-a^2h,
-h_x-b^2P_\beta u
\right)^\top.
\]

We first record the basic properties of the damping operator.

\begin{proposition}\label{prop:damping}
The operator
$
\mathcal D:
\mathcal H
\longrightarrow
\mathcal H
$
is bounded, self-adjoint, and nonnegative. Moreover,
\[
\langle
\mathcal D Z,Z
\rangle_{\mathcal H}
=
\|\mathcal B^*Z\|_{\mathcal U}^2,
\qquad
Z\in\mathcal H.
\]
\end{proposition}

\begin{proof}
Since
$
\mathcal D=\mathcal B\mathcal B^*
$
and
\(\mathcal B\in\mathcal L(\mathcal U,\mathcal H)\), we have
\(\mathcal D\in\mathcal L(\mathcal H)\). Moreover,
\[
\mathcal D^*
=
(\mathcal B\mathcal B^*)^*
=
(\mathcal B^*)^*\mathcal B^*
=
\mathcal B\mathcal B^*
=
\mathcal D.
\]
Thus, $\mathcal D$ is self-adjoint. Moreover,
\[
\langle \mathcal D Z,Z\rangle_{\mathcal H}
=
\langle \mathcal B\mathcal B^*Z,Z\rangle_{\mathcal H}
=
\|\mathcal B^*Z\|_{\mathcal U}^2
\ge 0,
\]
so that $\mathcal D$ is nonnegative.
\end{proof}

Since \(\mathcal D\) is bounded, the well-posedness of the damped
dynamics follows directly from the conservative theory.

\begin{theorem}
The operator $\mathcal A_d
=
\mathcal A-\mathcal D
$
generates a strongly continuous semigroup
\(\{S_d(t)\}_{t\ge0}\) on \(\mathcal H\). Consequently, for every
\(Z_0\in\mathcal H\), system \eqref{dp-s} admits a unique mild solution
\[
Z(t)
=
S_d(t)Z_0
\in
C([0,\infty);\mathcal H).
\]
If, in addition,
\(Z_0\in D(\mathcal A_d)=D(\mathcal A)\), then
\[
Z
\in
C([0,\infty);D(\mathcal A_d))
\cap
C^1([0,\infty);\mathcal H),
\]
and \(Z\) is the unique classical solution of \eqref{dp-s}.
\end{theorem}

\begin{proof}
By Theorem~\ref{thm:well-posedness-conservative}, \(\mathcal A\)
generates a strongly continuous unitary group on \(\mathcal H\).
Since $\mathcal D\in\mathcal L(\mathcal H)$,
the Bounded Perturbation Theorem
(see, e.g., \cite[Theorem~3.1.1]{Pazy}) implies that
\(\mathcal A_d=\mathcal A-\mathcal D\) generates a strongly continuous
semigroup on \(\mathcal H\). The existence, uniqueness, and regularity
statements then follow from the standard theory of
\(C_0\)-semigroups.
\end{proof}

We conclude this section with the energy dissipation identity.

\begin{proposition}
Let $Z(t)=(h(t),u(t))^\top$ be the classical solution of \eqref{dp-s} corresponding to
\(Z_0\in D(\mathcal A_d)\). Then
\begin{equation}\label{eq:energy-dissipation}
\frac12
\frac{d}{dt}
\|Z(t)\|_{\mathcal H}^2
=
-\|ah(t)\|_{L^2(\mathbb T)}^2
-
\|bP_\beta u(t)\|_{L^2(\mathbb T)}^2.
\end{equation}
Consequently,
\[
\|S_d(t)Z_0\|_{\mathcal H}
\le
\|Z_0\|_{\mathcal H},
\qquad
t\ge0,
\]
for every \(Z_0\in\mathcal H\). In particular,
\(\{S_d(t)\}_{t\ge0}\) is a contraction semigroup.
\end{proposition}

\begin{proof} Let \(Z\) be a classical solution. Then
\[
\frac12
\frac{d}{dt}
\|Z(t)\|_{\mathcal H}^2
=
\operatorname{Re}
\langle
\mathcal A_dZ(t),Z(t)
\rangle_{\mathcal H}.
\]
Using $\mathcal A_d=\mathcal A-\mathcal D$ and the skew-adjointness of \(\mathcal A\), we obtain
\[
\operatorname{Re}
\langle
\mathcal A_dZ(t),Z(t)
\rangle_{\mathcal H}
=
-
\langle
\mathcal DZ(t),Z(t)
\rangle_{\mathcal H}.
\]
By Proposition~\ref{prop:damping} and since $\mathcal B^*Z(t)=\left(ah(t),bP_\beta u(t)\right)^\top$, we have
\[
\frac12
\frac{d}{dt}
\|Z(t)\|_{\mathcal H}^2
=
-
\|\mathcal B^*Z(t)\|_{\mathcal U}^2,
\]
giving \eqref{eq:energy-dissipation}. Now, integrating in time gives
\[
\|S_d(t)Z_0\|_{\mathcal H}
\le
\|Z_0\|_{\mathcal H},
\qquad
t\ge0,
\]
for \(Z_0\in D(\mathcal A_d)\). Finally, the estimate extends to every \(Z_0\in\mathcal H\) by continuity of the semigroup, because \(D(\mathcal A_d)\) is dense in \(\mathcal H\).
\end{proof}


\section{Spectral analysis of the conservative operator}
\label{sec:spectral-analysis}

In this section, we characterize the spectral structure of the
conservative operator \(\mathcal A\). Since the spatial domain is the
periodic torus \(\mathbb T\), the eigenvalue problem can be analyzed
mode by mode through the Fourier basis, reducing it to a family of
two-dimensional matrix problems.

\subsection{Fourier reduction and eigenvalues}

We consider the eigenvalue problem
\begin{equation}\label{eq:eigenvalue-problem}
\mathcal A U=\mu U,
\qquad
U=(h,u)^\top\in D(\mathcal A)\setminus\{0\},
\end{equation}
where \(\mu\in\mathbb C\). Recall that
\[
\mathcal A(h,u)
=
\left(
T_\beta u,-h_x
\right)^\top,
\]
with
\[
\widehat{T_\beta u}(k)
=
-im_\beta(k)\widehat u(k),
\qquad
m_\beta(k)
=
\frac{\tanh(\beta k)}{\beta}.
\]

Expanding
\[
h(x)
=
\sum_{k\in\mathbb Z}\widehat h(k)e^{ikx},
\qquad
u(x)
=
\sum_{k\in\mathbb Z}\widehat u(k)e^{ikx},
\]
we obtain
\[
\widehat{\mathcal AU}(k)
=
\begin{pmatrix}
-im_\beta(k)\widehat u(k)
\\[1mm]
-ik\widehat h(k)
\end{pmatrix}.
\]
Therefore, \eqref{eq:eigenvalue-problem} is equivalent, for every
\(k\in\mathbb Z\), to
\[
A(k)\widehat U(k)
=
\mu\widehat U(k),
\qquad
\widehat U(k)
=
\begin{pmatrix}
\widehat h(k)
\\
\widehat u(k)
\end{pmatrix},
\]
where
\[
A(k)
=
\begin{pmatrix}
0 & -im_\beta(k)
\\
-ik & 0
\end{pmatrix}.
\]

For each \(k\in\mathbb Z\),
\[
\mu I-A(k)
=
\begin{pmatrix}
\mu & im_\beta(k)
\\
ik & \mu
\end{pmatrix},
\]
and hence
\[
\det(\mu I-A(k))
=
\mu^2+km_\beta(k).
\]
We therefore introduce the sequence
\begin{equation}\label{eq:dispersion-sequence}
\omega_k
:=
\sqrt{km_\beta(k)}
=
\sqrt{\frac{k\tanh(\beta k)}{\beta}},
\qquad
k\in\mathbb Z.
\end{equation}
We know that
\[
km_\beta(k)
=
\frac{k\tanh(\beta k)}{\beta}
\ge0,
\]
thus the sequence \((\omega_k)_{k\in\mathbb Z}\) is well defined. The
characteristic equation becomes
\[
\mu^2+\omega_k^2=0,
\]
and therefore the eigenvalues of \(A(k)\) are
\[
\mu_k^\pm
=
\pm i\omega_k.
\]

The relevant properties of the dispersion sequence are summarized
below.

\begin{proposition}\label{prop:eigenvalue-properties}
The sequence \((\omega_k)_{k\in\mathbb Z}\) defined by
\eqref{eq:dispersion-sequence} satisfies:

\begin{enumerate}
\item[(i)] We have that  $\omega_0=0$ and \(A(0)=0\).

\item[(ii)]
For every \(k\neq0\),
\[
\omega_k>0,
\]
and the eigenvalues \(\pm i\omega_k\) of \(A(k)\) are simple.

\item[(iii)]
The sequence is even:
\[
\omega_{-k}
=
\omega_k,
\qquad
k\in\mathbb Z.
\]

\item[(iv)]
The positive frequencies are strictly increasing:
\[
0<\omega_1<\omega_2<\cdots.
\]

\item[(v)]
The high-frequency asymptotic behavior is
\[
\omega_k
\sim
\frac{1}{\sqrt{\beta}}|k|^{1/2},
\qquad
|k|\to\infty.
\]
\end{enumerate}
\end{proposition}

\begin{proof}
Property (i) follows immediately from \(m_\beta(0)=0\).

For \(k\neq0\), the quantities \(k\) and \(\tanh(\beta k)\) have the
same sign. Hence
\[
k\tanh(\beta k)>0,
\]
which proves that \(\omega_k>0\). Note that the characteristic polynomial
\[
\mu^2+\omega_k^2
\]
has the two distinct roots \(\pm i\omega_k\), therefore these eigenvalues are
simple for each fixed \(k\neq0\). This proves (ii).

Since
\[
(-k)\tanh(-\beta k)
=
k\tanh(\beta k),
\]
we obtain
\[
\omega_{-k}=\omega_k,
\]
which proves (iii).

To prove (iv), define
\[
f(x)
=
\frac{x\tanh(\beta x)}{\beta},
\qquad
x>0.
\]
Then
\[
f'(x)
=
\frac{
\tanh(\beta x)
+
\beta x\,\operatorname{sech}^2(\beta x)
}{\beta}
>0,
\qquad
x>0.
\]
Thus \(f\) is strictly increasing on \((0,\infty)\). Since the
square-root function is strictly increasing on \([0,\infty)\), it
follows that
\[
0<\omega_1<\omega_2<\cdots.
\]

Finally, by using that
\[
\tanh(\beta|k|)
\longrightarrow1,
\qquad
\text{as }|k|\to\infty,
\]
we have
\[
\frac{\omega_k}
{\beta^{-1/2}|k|^{1/2}}
=
\sqrt{\tanh(\beta|k|)}
\longrightarrow1,
\]
which proves (v).
\end{proof}

\subsection{Compact resolvent and characterization of the spectrum}

We now show that the conservative operator \(\mathcal A\) has compact
resolvent. This property, together with the Fourier reduction above,
allows us to characterize its spectrum completely.

\begin{proposition}\label{prop:compact-resolvent}
The operator \(\mathcal A\) has compact resolvent.
\end{proposition}

\begin{proof}
Let \(\lambda>0\). By Proposition~\ref{prop:resolvent},
\[
(\lambda I-\mathcal A)^{-1}
\in
\mathcal L(\mathcal H,D(\mathcal A)).
\]
Moreover, the embeddings $H^{1/2}(\mathbb T)
\hookrightarrow
L^2(\mathbb T)$ and  $L^2(\mathbb T)
\hookrightarrow
H^{-1/2}(\mathbb T)$
are compact. Consequently,
\[
D(\mathcal A)
=
H^{1/2}(\mathbb T)\times L^2(\mathbb T)
\hookrightarrow
L^2(\mathbb T)\times H^{-1/2}(\mathbb T)
=
\mathcal H
\]
is compact. Moreover, again by Proposition \ref{prop:resolvent}, we have
\[
(\lambda I-\mathcal A)^{-1}
:
\mathcal H
\longrightarrow
D(\mathcal A)
\]
is bounded; its composition with the compact embedding
\(D(\mathcal A)\hookrightarrow\mathcal H\) is compact. Therefore,
\[
(\lambda I-\mathcal A)^{-1}
\in
\mathcal K(\mathcal H),
\]
and \(\mathcal A\) has compact resolvent.
\end{proof}

We can now give the complete spectral characterization.

\begin{theorem}
The spectrum of the conservative operator \(\mathcal A\) is purely
discrete and is given by
\[
\sigma(\mathcal A)
=
\{0\}
\cup
\left\{
\pm i\omega_n:
n\in\mathbb N^*
\right\},
\]
where \(\omega_n\) is defined by
\eqref{eq:dispersion-sequence}. Moreover,
\[
\ker(\mathcal A)
=
\operatorname{span}
\left\{
\binom{1}{0},
\binom{0}{1}
\right\},
\]
and, for every \(n\in\mathbb N^*\),
\begin{equation}\label{eq:eigenspaces-conservative}
\ker(\pm i\omega_n I-\mathcal A)
=
\left\{
\begin{pmatrix}
c_+e^{inx}+c_-e^{-inx}
\\[2mm]
\displaystyle
\mp\frac{n}{\omega_n}
\left(
c_+e^{inx}-c_-e^{-inx}
\right)
\end{pmatrix}
:
c_+,c_-\in\mathbb C
\right\}.
\end{equation}
In particular, \(0\) and each eigenvalue
\(\pm i\omega_n\), \(n\in\mathbb N^*\), have geometric multiplicity
two. Since \(\mathcal A\) is normal, their algebraic multiplicities
coincide with their geometric multiplicities. Furthermore, the
spectrum has no finite accumulation point.
\end{theorem}

\begin{proof}
By Proposition~\ref{prop:compact-resolvent}, the operator
\(\mathcal A\) has compact resolvent. Hence its spectrum consists only
of isolated eigenvalues of finite algebraic multiplicity and has no
finite accumulation point.

Suppose that $\mathcal AU=\mu U$, with $U\in D(\mathcal A)\setminus\{0\}$. For every \(k\in\mathbb Z\), the Fourier coefficients satisfy
\[
A(k)\widehat U(k)
=
\mu\widehat U(k).
\]
Since \(U\neq0\), there exists at least one \(k\in\mathbb Z\) such that
\(\widehat U(k)\neq0\). Therefore, \(\mu\) must be an eigenvalue of
\(A(k)\), and hence
\[
\mu=0
\qquad\text{or}\qquad
\mu=\pm i\omega_n
\]
for some \(n\in\mathbb N^*\).

Conversely, for \(k=0\), the equation
\[
A(0)\widehat U(0)=0
\]
shows that
\[
\ker(\mathcal A)
=
\operatorname{span}
\left\{
\binom{1}{0},
\binom{0}{1}
\right\}.
\]

Now fix \(n\in\mathbb N^*\). Since
\(\omega_{-n}=\omega_n\), the eigenvalue \(i\omega_n\) arises from the
two Fourier modes \(n\) and \(-n\). Solving
\[
A(k)\widehat U(k)
=
i\omega_n\widehat U(k)
\]
for \(k=\pm n\) gives
\[
\widehat u(n)
=
-\frac{n}{\omega_n}\widehat h(n),
\qquad
\widehat u(-n)
=
\frac{n}{\omega_n}\widehat h(-n).
\]
Thus
\[
\ker(i\omega_n I-\mathcal A)
=
\left\{
\begin{pmatrix}
c_+e^{inx}+c_-e^{-inx}
\\[2mm]
\displaystyle
-\frac{n}{\omega_n}
\left(
c_+e^{inx}-c_-e^{-inx}
\right)
\end{pmatrix}
:
c_+,c_-\in\mathbb C
\right\}.
\]
The case \(-i\omega_n\) is obtained analogously, yielding
\eqref{eq:eigenspaces-conservative}.

Hence every eigenvalue has geometric multiplicity two. By Theorem~\ref{thm:skew-adjoint-A}, the operator \(\mathcal A\) is skew-adjoint and therefore normal. Hence its algebraic and geometric multiplicities coincide.

Finally, Proposition~\ref{prop:eigenvalue-properties} gives
\[
\omega_n
\sim
\frac1{\sqrt{\beta}}n^{1/2}
\longrightarrow\infty,
\]
so the spectrum has no finite accumulation point.
\end{proof}

\section{Spectral properties of the damped generator}
\label{sec:spectral-damped}

We now turn to the spectral analysis of the damped generator. Recall
from Subsection~\ref{sec-damped} that
\[
\mathcal A_d
=
\mathcal A-\mathcal D,
\qquad
D(\mathcal A_d)=D(\mathcal A),
\]
where
\[
\mathcal A
\binom{h}{u}
=
\binom{T_\beta u}{-h_x},
\qquad
D(\mathcal A)
=
H^{1/2}(\mathbb T)\times L^2(\mathbb T),
\]
and
\[
\mathcal D
=
\mathcal B\mathcal B^*.
\]
Thus,
\[
\mathcal A_d
\binom{h}{u}
=
\binom{
T_\beta u-a^2h
}{
-h_x-b^2P_\beta u
}.
\]
By Section~\ref{sec:well-posedness}, the operator \(\mathcal A_d\)
generates a contraction semigroup on \(\mathcal H\) and is therefore
maximal dissipative. In particular,
\[
\{z\in\mathbb C:\operatorname{Re}z>0\}
\subset
\rho(\mathcal A_d).
\]

The purpose of this section is to establish the compactness of the
resolvent of \(\mathcal A_d\), analyze the spectrum on the imaginary
axis, and deduce strong stability of the damped semigroup. These
properties will also provide the spectral framework for the
high-frequency resolvent analysis developed later.

\subsection{Compactness of the resolvent}

We now show that the compact-resolvent property of the conservative
generator is preserved under the bounded damping perturbation.

\begin{proposition}\label{prop:compact-resolvent-damped}
The damped generator \(\mathcal A_d\) has compact resolvent.
\end{proposition}

\begin{proof}
Fix \(\lambda>0\). Thanks to the fact that \(\mathcal A\) is skew-adjoint and
\(\mathcal A_d\) is maximal dissipative, we get
\[
\lambda\in\rho(\mathcal A)\cap\rho(\mathcal A_d).
\]
Set
\[
R(\lambda,\mathcal A)
=
(\lambda I-\mathcal A)^{-1}.
\]
Since
\[
\mathcal A_d
=
\mathcal A-\mathcal D,
\]
we have, on \(D(\mathcal A)\),
\[
\lambda I-\mathcal A_d
=
\lambda I-\mathcal A+\mathcal D
=
\left[
I+\mathcal D R(\lambda,\mathcal A)
\right]
(\lambda I-\mathcal A).
\]
Because
\[
\lambda I-\mathcal A_d
=
\left[I+\mathcal D R(\lambda,\mathcal A)\right]
(\lambda I-\mathcal A)
\]
and both \(\lambda I-\mathcal A\) and
\(\lambda I-\mathcal A_d\) are isomorphisms from
\(D(\mathcal A)\) onto \(\mathcal H\), the operator
\[
I+\mathcal D R(\lambda,\mathcal A):
\mathcal H\longrightarrow\mathcal H
\]
is an isomorphism. Therefore,
\[
(\lambda I-\mathcal A_d)^{-1}
=
R(\lambda,\mathcal A)
\left[
I+\mathcal D R(\lambda,\mathcal A)
\right]^{-1}.
\]
By Proposition~\ref{prop:compact-resolvent},
\[
R(\lambda,\mathcal A)
\in
\mathcal K(\mathcal H),
\]
while
\[
\left[
I+\mathcal D R(\lambda,\mathcal A)
\right]^{-1}
\in
\mathcal L(\mathcal H).
\]
Hence $(\lambda I-\mathcal A_d)^{-1}$ is compact as the composition of a compact operator with a bounded
operator. Thus \(\mathcal A_d\) has a compact resolvent.
\end{proof}

The compactness of the resolvent gives the following spectral
consequence.

\begin{corollary}\label{cor:discrete-spectrum-damped}
The spectrum of \(\mathcal A_d\) consists entirely of isolated
eigenvalues of finite algebraic multiplicity. Its only possible
accumulation point is infinity. Moreover,
\begin{equation}\label{sigma-a}
\sigma(\mathcal A_d)
\subset
\left\{
z\in\mathbb C:
\operatorname{Re}z\le0
\right\}.
\end{equation}
\end{corollary}

\begin{proof}
The first assertions follow from
Proposition~\ref{prop:compact-resolvent-damped} and the standard
spectral theory of operators with compact resolvent. Since \(\mathcal A_d\) is maximal dissipative, $\{z\in\mathbb C:\operatorname{Re}z>0\}
\subset
\rho(\mathcal A_d)$,
and therefore \eqref{sigma-a} follows.
\end{proof}

\subsection{Eigenvalues on the imaginary axis}

We next investigate the possible intersection of the spectrum of
\(\mathcal A_d\) with the imaginary axis. The dissipative structure
reduces this question to the corresponding spectral problem for the
conservative generator.

\begin{proposition}\label{prop:imaginary-kernel-reduction}
For every \(\tau\in\mathbb R\),
\begin{equation}\label{eq:imaginary-kernel-identity}
\ker(i\tau I-\mathcal A_d)
=
\ker(i\tau I-\mathcal A)
\cap
\ker(\mathcal B^*).
\end{equation}
\end{proposition}

\begin{proof}
Let $Z\in\ker(i\tau I-\mathcal A_d)$. Then
\[
\mathcal A_dZ=i\tau Z.
\]
Taking the real part of the inner product with \(Z\), we obtain
\[
0
=
\operatorname{Re}
\langle i\tau Z,Z\rangle_{\mathcal H}
=
\operatorname{Re}
\langle\mathcal A_dZ,Z\rangle_{\mathcal H}.
\]
Since
\[
\mathcal A_d
=
\mathcal A-\mathcal B\mathcal B^*
\]
and \(\mathcal A\) is skew-adjoint,
\[
\operatorname{Re}
\langle\mathcal A_dZ,Z\rangle_{\mathcal H}
=
-
\|\mathcal B^*Z\|_{\mathcal U}^2.
\]
Hence
\[
\mathcal B^*Z=0.
\]
Consequently,
\[
\mathcal DZ
=
\mathcal B\mathcal B^*Z
=
0,
\]
and therefore
\[
\mathcal AZ
=
\mathcal A_dZ+\mathcal DZ
=
i\tau Z.
\]
Thus
\[
Z
\in
\ker(i\tau I-\mathcal A)
\cap
\ker(\mathcal B^*).
\]

Conversely, suppose that
\[
Z
\in
\ker(i\tau I-\mathcal A)
\cap
\ker(\mathcal B^*).
\]
Then \(\mathcal DZ=0\), and hence
\[
\mathcal A_dZ
=
\mathcal AZ-\mathcal DZ
=
i\tau Z.
\]
This proves \eqref{eq:imaginary-kernel-identity}.
\end{proof}

The previous proposition is particularly useful because the
eigenspaces of the conservative generator were explicitly determined
in Section~\ref{sec:spectral-analysis}.

\begin{theorem}\label{thm:no-imaginary-eigenvalues-damped}
Assume that $a,b\in C^\infty(\mathbb T;\mathbb R)$ with $a,b\ge0$, $a\not\equiv0$, and  $b\not\equiv0$. Then
\begin{equation*}
\ker(i\tau I-\mathcal A_d)
=
\{0\},
\qquad
\tau\in\mathbb R.
\end{equation*}
In particular, \(\mathcal A_d\) has no eigenvalues on the imaginary
axis.
\end{theorem}

\begin{proof}
Let $Z=(h,u)^\top
\in
\ker(i\tau I-\mathcal A_d)$.   By Proposition~\ref{prop:imaginary-kernel-reduction}, $\mathcal AZ=i\tau Z$ and
\begin{equation}\label{eq:Bstar-zero-imaginary}
\mathcal B^*Z
=
(ah,bP_\beta u)^\top
=
0.
\end{equation}
We distinguish the cases \(\tau=0\) and \(\tau\neq0\).

\medskip
\noindent
\textbf{Case 1: \(\tau=0\).}

By the spectral characterization of \(\mathcal A\),
\[
\ker(\mathcal A)
=
\operatorname{span}
\left\{
\binom{1}{0},
\binom{0}{1}
\right\}.
\]
Hence
\[
h(x)=c_1,
\qquad
u(x)=c_2,
\]
for some \(c_1,c_2\in\mathbb C\). Note that \(P_\beta\) acts as the
identity on the zero Fourier mode, thus
\[
P_\beta u=c_2.
\]
Condition \eqref{eq:Bstar-zero-imaginary} gives
\[
a(x)c_1=0,
\qquad
b(x)c_2=0.
\]
As \(a\not\equiv0\) and \(b\not\equiv0\), it follows that $c_1=c_2=0$. Thus $Z=0$.

\medskip
\noindent
\textbf{Case 2: \(\tau\neq0\).}

Since \(Z\) is an eigenvector of \(\mathcal A\), the spectral
characterization obtained in Section~\ref{sec:spectral-analysis}
implies that
\[
\tau=\pm\omega_n
\]
for some \(n\in\mathbb N^*\). If \(\tau=\omega_n\), the description of the eigenspace gives
\[
h(x)
=
c_+e^{inx}
+
c_-e^{-inx}
\]
for some \(c_+,c_-\in\mathbb C\). The same form holds when
\(\tau=-\omega_n\).

We have that \(a\) is continuous, nonnegative, and not identically zero; thus the set
\[
\{x\in\mathbb T:a(x)>0\}
\]
contains a nonempty open subset \(\omega_1\).
Thanks to \eqref{eq:Bstar-zero-imaginary}, $ah=0$, and therefore
\[
h=0
\qquad
\text{on }\omega_1.
\]
The function \(h\) is a trigonometric polynomial; thus it is analytic and hence
cannot vanish on a nonempty open set unless it vanishes identically. So,  $h\equiv0$.  Taking into account that \(\mathcal AZ=i\tau Z\), the second component of the eigenvalue equation gives
\[
-h_x=i\tau u.
\]
As \(\tau\neq0\), it follows that \(u\equiv0\). Hence \(Z=0\). Therefore,
\[
\ker(i\tau I-\mathcal A_d)
=
\{0\}
\]
for every \(\tau\in\mathbb R\).
\end{proof}

The compactness of the resolvent now allows us to pass from the absence
of imaginary eigenvalues to the absence of spectrum on the imaginary
axis.

\begin{corollary}\label{cor:imaginary-axis-resolvent-damped}
The imaginary axis belongs to the resolvent set of the damped
generator:
\[
i\mathbb R
\subset
\rho(\mathcal A_d).
\]
Consequently,
\begin{equation}\label{eq:spectrum-strict-left-half-plane}
\sigma(\mathcal A_d)
\subset
\left\{
z\in\mathbb C:
\operatorname{Re}z<0
\right\}.
\end{equation}
\end{corollary}

\begin{proof}
By Proposition~\ref{prop:compact-resolvent-damped},
\(\mathcal A_d\) has compact resolvent. Hence every point of
\(\sigma(\mathcal A_d)\) is an isolated eigenvalue of finite algebraic
multiplicity. By
Theorem~\ref{thm:no-imaginary-eigenvalues-damped}, \(\mathcal A_d\)
has no eigenvalues on \(i\mathbb R\). Therefore,
\[
\sigma(\mathcal A_d)\cap i\mathbb R
=
\varnothing,
\]
which is equivalent to $i\mathbb R\subset\rho(\mathcal A_d)$.  On the other hand, Corollary~\ref{cor:discrete-spectrum-damped} implies
\[
\sigma(\mathcal A_d)
\subset
\{z\in\mathbb C:\operatorname{Re}z\le0\}.
\]
Combining the two assertions yields \eqref{eq:spectrum-strict-left-half-plane}.
\end{proof}

\begin{remark}
The strict spectral inclusion
\[
\sigma(\mathcal A_d)
\subset
\{z\in\mathbb C:\operatorname{Re}z<0\}
\]
does not by itself imply exponential stability. Indeed, since
\(S_d(t)\) is a bounded semigroup on a Hilbert space, the
Gearhart--Pr\"uss theorem \cite{Gearhart1978,Pruss1984} requires, in addition to
\(i\mathbb R\subset\rho(\mathcal A_d)\), the uniform resolvent bound
\[
\sup_{\tau\in\mathbb R}
\left\|
(i\tau I-\mathcal A_d)^{-1}
\right\|_{\mathcal L(\mathcal H)}
<\infty.
\]
The latter property is not a consequence of the strict spectral
inclusion above. The high-frequency behavior of the resolvent will
therefore play a decisive role in determining the decay rate of the
damped semigroup.
\end{remark}

\subsection{Strong stability of the damped semigroup}

We now establish the asymptotic stability of the damped semigroup.
Recall that \(\mathcal A_d\) generates a contraction semigroup
\(\{S_d(t)\}_{t\ge0}\) on \(\mathcal H\). Hence
\[
\|S_d(t)\|_{\mathcal L(\mathcal H)}
\le1,
\qquad
t\ge0.
\]
Moreover, the spectral analysis of the previous subsection gives
\[
i\mathbb R
\subset
\rho(\mathcal A_d),
\]
and therefore
\[
\sigma(\mathcal A_d)\cap i\mathbb R
=
\varnothing.
\]

We recall the following classical stability criterion.

\begin{theorem}[Arendt--Batty--Lyubich--V\~u
\cite{ArendtBatty,LyubichVu}]
\label{thm:ABLV}
Let \(A\) be the generator of a bounded \(C_0\)-semigroup
\(\{S(t)\}_{t\ge0}\) on a Banach space \(X\). Assume that
\[
\sigma(A)\cap i\mathbb R
\]
is at most countable and that
\[
\sigma_p(A^*)\cap i\mathbb R
=
\varnothing.
\]
Then
\[
\lim_{t\to\infty}
\|S(t)x\|_X
=
0,
\qquad
x\in X.
\]
\end{theorem}

We can now deduce the strong stability of the damped dynamics.

\begin{theorem}\label{thm:strong-stability-damped}
Let \(S_d(t)\) be the contraction semigroup generated by
\(\mathcal A_d\). Then
\[
\lim_{t\to\infty}
\|S_d(t)Z_0\|_{\mathcal H}
=
0,
\qquad
Z_0\in\mathcal H.
\]
\end{theorem}
\begin{proof}
Since \(S_d(t)\) is a contraction semigroup, it satisfies
\[
\sup_{t\ge0}\|S_d(t)\|_{\mathcal L(\mathcal H)}\le1.
\]
Moreover, by Corollary~\ref{cor:imaginary-axis-resolvent-damped},
\[
i\mathbb R\subset\rho(\mathcal A_d),
\]
and hence
\[
\sigma(\mathcal A_d)\cap i\mathbb R=\varnothing.
\]
In particular, this intersection is at most countable.

For every \(\tau\in\mathbb R\), the operator
\(i\tau I-\mathcal A_d\) is invertible. Taking adjoints gives
\[
-i\tau I-\mathcal A_d^*
\]
invertible, with
\[
(-i\tau I-\mathcal A_d^*)^{-1}
=
\left[
(i\tau I-\mathcal A_d)^{-1}
\right]^*.
\]
Note that \(\tau\in\mathbb R\) is arbitrary, thus
\[
i\mathbb R\subset\rho(\mathcal A_d^*),
\]
and therefore
\[
\sigma_p(\mathcal A_d^*)\cap i\mathbb R=\varnothing.
\]
The Arendt--Batty--Lyubich--V\~u theorem,
Theorem~\ref{thm:ABLV}, now yields
\[
\lim_{t\to\infty}
\|S_d(t)Z_0\|_{\mathcal H}
=
0,
\]
for every \(Z_0\in\mathcal H\).
\end{proof}

In terms of the energy, strong stability immediately gives the
following consequence.

\begin{corollary}
Let
\[
Z(t)
=
S_d(t)Z_0
=
\binom{h(t)}{u(t)},
\qquad
Z_0\in\mathcal H.
\]
Then $E(t)\longrightarrow0$,  as $t\to\infty$, where
\[
E(t)
=
\frac12
\|h(t)\|_{L^2(\mathbb T)}^2
+
\frac12
\langle
P_\beta u(t),u(t)
\rangle_{H^{1/2},H^{-1/2}}.
\]
\end{corollary}

\begin{proof}
By the definition of the energy inner product on \(\mathcal H\),
\[
E(t)
=
\frac12
\|Z(t)\|_{\mathcal H}^2.
\]
The conclusion therefore follows immediately from
Theorem~\ref{thm:strong-stability-damped}.
\end{proof}

\begin{remark}
Theorem~\ref{thm:strong-stability-damped} is qualitative and does not
provide a uniform rate of decay. Although $i\mathbb R
\subset
\rho(\mathcal A_d)$,
this does not imply that
\[
\sup_{\tau\in\mathbb R}
\left\|
(i\tau I-\mathcal A_d)^{-1}
\right\|_{\mathcal L(\mathcal H)}
<
\infty.
\]
Thus, strong stability alone does not imply exponential stability.
Determining the decay rate requires a quantitative analysis of
\[
(i\tau I-\mathcal A_d)^{-1}
\]
as \(|\tau|\to\infty\). This is the purpose of the next section.
\end{remark}

\section{High-frequency behavior of the resolvent}
\label{sec:high-frequency-resolvent}

Having established the strong stability of the damped semigroup, we
now turn to the quantitative behavior of the resolvent on the
imaginary axis. At this stage, no growth estimate is assumed for $(i\tau I-\mathcal A_d)^{-1}$
as \(|\tau|\to\infty\). Our first objective is to determine whether
the resolvent can remain uniformly bounded at high frequencies.

Recall from Corollary~\ref{cor:imaginary-axis-resolvent-damped} that $i\mathbb R\subset\rho(\mathcal A_d)$.  Since the resolvent depends continuously on the spectral parameter, it is uniformly bounded on every compact subset of the imaginary axis. Thus, the question of uniform boundedness reduces entirely to the
regime $|\tau|\to\infty$.

We first show that uniform boundedness fails when the damping acting on
the first component is genuinely localized. More precisely, for the
quasimode construction below, we assume that there exists a nonempty
open set
\begin{equation}\label{eq:undamped-open-set}
\mathcal O
\Subset
\mathbb T\setminus\operatorname{supp}(a).
\end{equation}
The high-frequency upper bound established later will not require this
additional localization assumption. Throughout this section, we use the energy norm of \(\mathcal H\),
namely
\[
\|Z\|_{\mathcal H}^2
=
\|h\|_{L^2(\mathbb T)}^2
+
\langle
P_\beta u,u
\rangle_{H^{1/2},H^{-1/2}},
\qquad
Z=(h,u)^\top.
\]

\subsection{Construction of high-frequency quasimodes}

Recall that the nonzero eigenvalues of the conservative generator are $\pm i\omega_k$, with $k\in\mathbb N^*$,
where
\[
\omega_k
=
\sqrt{\frac{k\tanh(\beta k)}{\beta}}.
\]
For every \(k\in\mathbb Z\setminus\{0\}\), define
\[
\Phi_k^+(x)
=
\begin{pmatrix}
1
\\[1mm]
-\dfrac{k}{\omega_{|k|}}
\end{pmatrix}
e^{ikx}.
\]
A direct computation gives
\begin{equation*}
\mathcal A\Phi_k^+
=
i\omega_{|k|}\Phi_k^+.
\end{equation*}
Indeed,
\[
\omega_{|k|}^2
=
k\,m_\beta(k),
\]
and therefore
\[
T_\beta
\left(
-\frac{k}{\omega_{|k|}}e^{ikx}
\right)
=
i\omega_{|k|}e^{ikx},
\]
while
\[
-\partial_xe^{ikx}
=
-ik e^{ikx}
=
i\omega_{|k|}
\left(
-\frac{k}{\omega_{|k|}}e^{ikx}
\right).
\]
Moreover, since
\[
p_\beta(k)
=
\frac{m_\beta(k)}{k},
\qquad
k\neq0,
\]
we have
\begin{equation}\label{eq:equipartition-eigenmode}
p_\beta(k)
\frac{k^2}{\omega_{|k|}^2}
=
1.
\end{equation}
Thus, the two components of each conservative eigenmode contribute
equally to the energy.

Choose a nonzero function $\chi\in C_c^\infty(\mathcal O)$ and define $h_N(x)=\chi(x)e^{iNx}$, for \(N\in\mathbb N^*\). By \eqref{eq:undamped-open-set},
\begin{equation}\label{eq:a-hN-zero}
ah_N=0
\qquad
\text{on }\mathbb T.
\end{equation}
Writing
\[
\chi(x)
=
\sum_{j\in\mathbb Z}
\widehat\chi(j)e^{ijx},
\]
we obtain
\[
h_N(x)
=
\sum_{k\in\mathbb Z}
\widehat\chi(k-N)e^{ikx}.
\]

We associate with \(h_N\) the second component \(u_N\), defined by
\begin{equation}\label{eq:def-uN-quasimode}
\widehat u_N(k)
=
-
\frac{k}{\omega_{|k|}}
\widehat h_N(k),
\qquad
k\neq0,
\end{equation}
and $\widehat u_N(0)=0$. We then set $ Z_N=(h_N,u_N)^\top$.  Since \(\chi\in C^\infty(\mathbb T)\), its Fourier coefficients decay
faster than any polynomial. Moreover,
\[
\frac{|k|}{\omega_{|k|}}
\lesssim_\beta
\langle k\rangle^{1/2},
\qquad k\neq0.
\]
It follows from \eqref{eq:def-uN-quasimode} that \(u_N\in L^2(\mathbb T)\)
(and in fact \(u_N\in C^\infty(\mathbb T)\)). Hence
\[
Z_N\in D(\mathcal A_d)=D(\mathcal A), \quad \forall N\in\mathbb N^*.
\]

\begin{lemma}\label{lem:quasimode-normalization}
There exist constants \(c,C>0\), independent of \(N\), such that
\[
c
\le
\|Z_N\|_{\mathcal H}
\le
C
\]
for all sufficiently large \(N\).
\end{lemma}

\begin{proof}
By Parseval's identity and
\eqref{eq:equipartition-eigenmode},
\begin{align*}
\langle P_\beta u_N,u_N\rangle_{H^{1/2},H^{-1/2}}
=
2\pi
\sum_{k\neq0}
p_\beta(k)
\frac{k^2}{\omega_{|k|}^2}
|\widehat h_N(k)|^2
=
2\pi
\sum_{k\neq0}
|\widehat h_N(k)|^2.
\end{align*}
Hence
\[
\|Z_N\|_{\mathcal H}^2
=
2\|h_N\|_{L^2(\mathbb T)}^2
-
2\pi|\widehat h_N(0)|^2.
\]
As
\[
\|h_N\|_{L^2(\mathbb T)}
=
\|\chi\|_{L^2(\mathbb T)}
\]
and $\widehat h_N(0)=\widehat\chi(-N)$, the smoothness of \(\chi\) implies
\[
|\widehat\chi(-N)|
=
O(N^{-M}),
\]
for every \(M>0\). Therefore,
\[
\|Z_N\|_{\mathcal H}^2
\longrightarrow
2\|\chi\|_{L^2(\mathbb T)}^2,
\qquad
\text{as }N\to\infty.
\]
Because \(\chi\neq0\), the conclusion follows.
\end{proof}

We next estimate the defect of \(Z_N\) as an approximate eigenvector of
the conservative generator at frequency \(\omega_N\).

\begin{lemma}
There exists \(C>0\) such that
\begin{equation}\label{eq:conservative-quasimode-error}
\|
(i\omega_N I-\mathcal A)Z_N
\|_{\mathcal H}
\le
\frac{C}{\sqrt N}
\end{equation}
for all sufficiently large \(N\).
\end{lemma}

\begin{proof}
For the nonzero Fourier modes,
\[
Z_N
=
\sum_{k\neq0}
\widehat h_N(k)\Phi_k^+
+
\widehat h_N(0)
\binom{1}{0}.
\]
Consequently,
\[
(i\omega_NI-\mathcal A)Z_N
=
i
\sum_{k\neq0}
\left(
\omega_N-\omega_{|k|}
\right)
\widehat h_N(k)\Phi_k^+
+
R_N^{(0)},
\]
where
\[
R_N^{(0)}
=
i\omega_N\widehat h_N(0)
\binom{1}{0}.
\]
Since
\[
\widehat h_N(0)=\widehat\chi(-N),
\qquad
\omega_N=O(N^{1/2}),
\]
the rapid decay of \(\widehat\chi\) gives
\[
\|R_N^{(0)}\|_{\mathcal H}
=
O(N^{-M})
\]
for every \(M>0\), after possibly changing \(M\).

For \(\xi>0\), set
\[
\omega(\xi)
=
\sqrt{
\frac{\xi\tanh(\beta\xi)}{\beta}
}.
\]
Then
\begin{equation*}
\omega'(\xi)
=
\frac{
\tanh(\beta\xi)
+
\beta\xi\,\operatorname{sech}^2(\beta\xi)
}{
2\beta\omega(\xi)
}
=
O(\xi^{-1/2})
\end{equation*}
as \(\xi\to\infty\).  We split the Fourier sum into
\[
|k-N|\le\frac N2
\qquad\text{and}\qquad
|k-N|>\frac N2.
\]

In the first region, \(k\ge N/2>0\), so \(k\asymp N\). By the mean
value theorem,
\[
|\omega_k-\omega_N|
\le
C\frac{|k-N|}{\sqrt N}.
\]
Therefore,
\begin{align*}
\sum_{|k-N|\le N/2}
|\omega_k-\omega_N|^2
|\widehat h_N(k)|^2
\le
\frac{C}{N}
\sum_{j\in\mathbb Z}
j^2|\widehat\chi(j)|^2
\le
\frac{C}{N}
\|\chi\|_{H^1(\mathbb T)}^2.
\end{align*}
Now, in the complementary region, the rapid decay of
\(\widehat\chi(k-N)\), together with
\[
\omega_{|k|}+\omega_N
\lesssim
1+|k|^{1/2}+N^{1/2},
\]
implies that, for every \(M>0\),
\[
\sum_{|k-N|>N/2}
|\omega_{|k|}-\omega_N|^2
|\widehat\chi(k-N)|^2
=
O(N^{-M}).
\]
Using the orthogonality of the Fourier modes together with
\eqref{eq:equipartition-eigenmode}, we conclude that
\[
\|
(i\omega_NI-\mathcal A)Z_N
\|_{\mathcal H}^2
\le
\frac{C}{N},
\]
which proves \eqref{eq:conservative-quasimode-error}.
\end{proof}

\subsection{Contribution of the damping}

We next estimate the contribution of the damping operator on the
quasimodes constructed above. This contribution is of lower order than
the conservative quasimode defect.

\begin{lemma}
There exists \(C>0\) such that
\begin{equation}\label{eq:damping-quasimode-error}
\|\mathcal D Z_N\|_{\mathcal H}
\le
\frac{C}{N}
\end{equation}
for all sufficiently large \(N\).
\end{lemma}

\begin{proof}
Recall that
\[
\mathcal D Z_N
=
\begin{pmatrix}
a^2h_N
\\
b^2P_\beta u_N
\end{pmatrix}.
\]
By \eqref{eq:a-hN-zero}, $a^2h_N=0$. It therefore remains to estimate the second component.

For \(k\neq0\), using
\eqref{eq:def-uN-quasimode}, we obtain
\begin{align*}
\widehat{P_\beta u_N}(k)
&=
-p_\beta(k)
\frac{k}{\omega_{|k|}}
\widehat h_N(k)
=
-\frac{m_\beta(k)}{\omega_{|k|}}
\widehat h_N(k)
=
-\frac{\omega_{|k|}}{k}
\widehat h_N(k),
\end{align*}
where we used $\omega_{|k|}^2 = k\,m_\beta(k)$. Consequently,
\[
\left|
\widehat{P_\beta u_N}(k)
\right|^2
=
p_\beta(k)
|\widehat h_N(k)|^2,
\qquad
k\neq0.
\]
Because
\[
p_\beta(k)
\lesssim_\beta
\langle k\rangle^{-1},
\]
we obtain
\begin{align*}
\|P_\beta u_N\|_{H^{-1/2}(\mathbb T)}^2
\asymp
\sum_{k\neq0}
\langle k\rangle^{-1}
p_\beta(k)
|\widehat h_N(k)|^2
\lesssim_\beta
\sum_{k\neq0}
\langle k\rangle^{-2}
|\widehat\chi(k-N)|^2.
\end{align*}

We split the sum into
\[
|k-N|\le\frac N2
\qquad\text{and}\qquad
|k-N|>\frac N2.
\]
In the first region, \(k\asymp N\), and therefore
\[
\sum_{|k-N|\le N/2}
\langle k\rangle^{-2}
|\widehat\chi(k-N)|^2
\le
\frac{C}{N^2}
\sum_{j\in\mathbb Z}
|\widehat\chi(j)|^2
\le
\frac{C}{N^2}.
\]
In the complementary region, the rapid decay of
\(\widehat\chi(k-N)\) gives
\[
\sum_{|k-N|>N/2}
\langle k\rangle^{-2}
|\widehat\chi(k-N)|^2
=
O(N^{-M})
\]
for every \(M>0\). Hence
\[
\|P_\beta u_N\|_{H^{-1/2}(\mathbb T)}
\le
\frac{C}{N}.
\]

Since multiplication by \(b^2\in C^\infty(\mathbb T)\) is bounded on
\(H^{-1/2}(\mathbb T)\),
\[
\|b^2P_\beta u_N\|_{H^{-1/2}(\mathbb T)}
\le
C
\|P_\beta u_N\|_{H^{-1/2}(\mathbb T)}
\le
\frac{C}{N}.
\]
Using the equivalence between the energy norm on \(\mathcal H\) and the
standard product norm on
\(L^2(\mathbb T)\times H^{-1/2}(\mathbb T)\), we conclude that
\[
\|\mathcal D Z_N\|_{\mathcal H}
\le
\frac{C}{N},
\]which proves the result.
\end{proof}

Combining \eqref{eq:conservative-quasimode-error} and \eqref{eq:damping-quasimode-error}, and using that  $\mathcal A_d=\mathcal A-\mathcal D$, we obtain 
\begin{equation}\label{eq:damped-quasimode-error}
\|
(i\omega_NI-\mathcal A_d)Z_N
\|_{\mathcal H}
\le
\frac{C}{\sqrt N}
\end{equation}
for all sufficiently large \(N\).

\subsection{Failure of uniform resolvent boundedness}

We are now in a position to rule out uniform boundedness of the
resolvent on the imaginary axis under the genuine localization
assumption \eqref{eq:undamped-open-set}.

\begin{theorem}\label{thm:resolvent-not-uniformly-bounded}
Assume that there exists a nonempty open set \(\mathcal O\) satisfying
\eqref{eq:undamped-open-set}. Then the resolvent of the damped generator
is not uniformly bounded on the imaginary axis. More precisely, there
exists \(c>0\) such that
\[
\|
(i\omega_NI-\mathcal A_d)^{-1}
\|_{\mathcal L(\mathcal H)}
\ge
c\sqrt N
\]
for all sufficiently large \(N\). Moreover, since
\[
\omega_N
\sim
\frac{1}{\sqrt\beta}\sqrt N,
\]
there exists \(c_\beta>0\) such that
\begin{equation}\label{eq:resolvent-linear-lower-bound}
\|
(i\omega_NI-\mathcal A_d)^{-1}
\|_{\mathcal L(\mathcal H)}
\ge
c_\beta\omega_N
\end{equation}
for all sufficiently large \(N\). In particular,
\begin{equation}\label{eq:resolvent-unbounded-imaginary-axis}
\sup_{\tau\in\mathbb R}
\|
(i\tau I-\mathcal A_d)^{-1}
\|_{\mathcal L(\mathcal H)}
=
+\infty.
\end{equation}
\end{theorem}

\begin{proof}
By Corollary~\ref{cor:imaginary-axis-resolvent-damped}, $i\omega_N\in\rho(\mathcal A_d)$, for every \(N\). Hence
\[
Z_N
=
(i\omega_NI-\mathcal A_d)^{-1}
(i\omega_NI-\mathcal A_d)Z_N,
\]
and therefore
\[
\|Z_N\|_{\mathcal H}
\le
\|
(i\omega_NI-\mathcal A_d)^{-1}
\|_{\mathcal L(\mathcal H)}
\,
\|
(i\omega_NI-\mathcal A_d)Z_N
\|_{\mathcal H}.
\]
By Lemma~\ref{lem:quasimode-normalization},
\[
\|Z_N\|_{\mathcal H}
\ge
c_0>0
\]
for all sufficiently large \(N\), whereas
\eqref{eq:damped-quasimode-error} gives
\[
\|
(i\omega_NI-\mathcal A_d)Z_N
\|_{\mathcal H}
\le
\frac{C}{\sqrt N}.
\]
It follows that
\[
\|
(i\omega_NI-\mathcal A_d)^{-1}
\|_{\mathcal L(\mathcal H)}
\ge
c\sqrt N.
\]
Finally, since
\[
\omega_N
\sim
\frac{1}{\sqrt\beta}\sqrt N,
\]
we obtain \eqref{eq:resolvent-linear-lower-bound}. Finally, this also yields
\eqref{eq:resolvent-unbounded-imaginary-axis} because \(\omega_N\to\infty\).
\end{proof}

As an immediate consequence, exponential stability is impossible under
genuine localization.

\begin{corollary}
Assume that \eqref{eq:undamped-open-set} holds. Then the damped
semigroup \(\{S_d(t)\}_{t\ge0}\) is not exponentially stable.
\end{corollary}

\begin{proof}
By the Gearhart--Pr\"uss theorem
\cite{Gearhart1978,Pruss1984}, a bounded \(C_0\)-semigroup on a Hilbert
space is exponentially stable if and only if  $i\mathbb R\subset\rho(\mathcal A_d)$ and
\[
\sup_{\tau\in\mathbb R}
\left\|
(i\tau I-\mathcal A_d)^{-1}
\right\|_{\mathcal L(\mathcal H)}
<\infty.
\]
In our setting, \(S_d(t)\) is a contraction semigroup on
\(\mathcal H\), and $i\mathbb R\subset\rho(\mathcal A_d)$,
by Corollary~\ref{cor:imaginary-axis-resolvent-damped}. However,
Theorem~\ref{thm:resolvent-not-uniformly-bounded} yields
\[
\sup_{\tau\in\mathbb R}
\left\|
(i\tau I-\mathcal A_d)^{-1}
\right\|_{\mathcal L(\mathcal H)}
=
+\infty.
\]
Therefore, \(S_d(t)\) cannot be exponentially stable.
\end{proof}

Under the same genuine-localization assumption, the lower bound
\eqref{eq:resolvent-linear-lower-bound} also shows that the exponent
\(1\) is necessary in any polynomial resolvent estimate.
Indeed, if
\[
\|
(i\tau I-\mathcal A_d)^{-1}
\|_{\mathcal L(\mathcal H)}
=
O(|\tau|^\alpha),
\qquad
\text{as }|\tau|\to\infty,
\]
then necessarily $\alpha\ge1$. We now turn to the corresponding upper bound of order \(|\tau|\).

\subsection{An upper bound for the high-frequency resolvent}

The lower bound obtained in
Theorem~\ref{thm:resolvent-not-uniformly-bounded} shows that, under the
genuine localization assumption \eqref{eq:undamped-open-set}, any
polynomial resolvent estimate must have exponent at least one. We now
establish the matching upper bound
\[
\left\|
(i\tau I-\mathcal A_d)^{-1}
\right\|_{\mathcal L(\mathcal H)}
\lesssim
|\tau|,
\qquad
|\tau|\to\infty.
\]
This upper bound does not require the additional localization
assumption \eqref{eq:undamped-open-set}. Combined with the quasimode
lower bound, it will show that the exponent \(1\) is sharp whenever the
damping is genuinely localized. We divide the proof into several steps.

\subsubsection*{\textbf{Step 1.} Reduction to an \(L^2\times L^2\) system}

Let $Z=(h,u)^\top\in D(\mathcal A_d)$ solve
\[
(i\tau I-\mathcal A_d)Z
=
F
=(f,g)^\top,
\qquad
\tau\in\mathbb R.
\]
In components, this system reads
\begin{equation}\label{eq:resolvent-upper-second}
\begin{cases}
i\tau h-T_\beta u+a^2h
=
f,
\\[1mm]
i\tau u+h_x+b^2P_\beta u
=
g.
\end{cases}
\end{equation}

We introduce the variable $v=P_\beta^{1/2}u$ and the Fourier multiplier \(K_\beta\) defined by
\[
\widehat{K_\beta w}(k)
=
\kappa_\beta(k)\widehat w(k),
\qquad
\kappa_\beta(k)
=
k\sqrt{p_\beta(k)}.
\]
Note that \(p_\beta\) is real, even, and strictly positive; thus
\(\kappa_\beta\) is real and odd. In particular, \(K_\beta\) is a
self-adjoint Fourier multiplier on \(L^2(\mathbb T)\), with natural
domain
\[
D(K_\beta)=H^{1/2}(\mathbb T).
\]
Moreover,
\[
K_\beta^2
=
T_\beta\partial_x,
\]
because
\[
\kappa_\beta(k)^2
=
k^2p_\beta(k)
=
\frac{k\tanh(\beta k)}{\beta}.
\]
We set
\[
L_\beta
:=
K_\beta^2
=
T_\beta\partial_x.
\]
Thus \(L_\beta\) is a nonnegative self-adjoint Fourier multiplier of
order one.

Using
\[
m_\beta(k)
=
kp_\beta(k),
\]
we obtain
\[
T_\beta u
=
-iK_\beta v,
\]
while
\[
P_\beta^{1/2}h_x
=
iK_\beta h.
\]
Therefore, applying \(P_\beta^{1/2}\) to the second equation in
\eqref{eq:resolvent-upper-second}, the system becomes
\begin{align}
i\tau h+iK_\beta v+a^2h
&=
f,
\label{eq:L2-resolvent-first}
\\
i\tau v+iK_\beta h+C_bv
&=
q,
\label{eq:L2-resolvent-second}
\end{align}
where
\[
q=P_\beta^{1/2}g,
\qquad
C_b
=
P_\beta^{1/2}b^2P_\beta^{1/2}.
\]
The operator \(C_b\) is bounded, self-adjoint, and nonnegative on
\(L^2(\mathbb T)\). Indeed, $C_b=C_b^*$,  and
\[
\langle C_bv,v\rangle_{L^2}
=
\|bP_\beta^{1/2}v\|_{L^2}^2
\ge0.
\]

For convenience, set
\[
W
=
\binom{h}{v},
\qquad
G
=
\binom{f}{q}.
\]
Since
\[
p_\beta(k)\asymp_\beta \langle k\rangle^{-1},
\]
the operator
\[
P_\beta^{1/2}:H^{-1/2}(\mathbb T)\longrightarrow L^2(\mathbb T)
\]
is an isomorphism. Moreover, by the definition of the energy inner
product on \(\mathcal H\),
\[
\|W\|_{L^2\times L^2}^2
=
\|h\|_{L^2}^2
+
\|P_\beta^{1/2}u\|_{L^2}^2
=
\|Z\|_{\mathcal H}^2.
\]
Similarly,
\[
\|G\|_{L^2\times L^2}^2
=
\|f\|_{L^2}^2
+
\|P_\beta^{1/2}g\|_{L^2}^2
=
\|F\|_{\mathcal H}^2.
\]
Thus the transformation
\[
(h,u)
\longmapsto
(h,P_\beta^{1/2}u)
\]
identifies \(\mathcal H\) isometrically with
\(L^2(\mathbb T)\times L^2(\mathbb T)\).

\subsubsection*{\textbf{Step 2.} A stationary scalar resolvent estimate}

The key ingredient in the high-frequency analysis is the following
stationary estimate.

\begin{lemma}
Assume that $a\in C^\infty(\mathbb T)$, $a\ge0$ and $a\not\equiv0$. There exist constants \(C>0\) and \(\tau_0>0\) such that
\begin{equation}\label{eq:scalar-stationary-resolvent}
\|w\|_{L^2(\mathbb T)}
\le
C
\left\|
\left(
L_\beta-\tau^2+i\tau a^2
\right)w
\right\|_{L^2(\mathbb T)}
\end{equation}
for every \(|\tau|\ge\tau_0\) and every
\(w\in H^1(\mathbb T)\). Moreover,
\begin{equation}\label{eq:scalar-stationary-K-resolvent}
\left\|
\left(
L_\beta-\tau^2+i\tau a^2
\right)^{-1}
K_\beta
\right\|_{\mathcal L(L^2(\mathbb T))}
\le
C|\tau|,
\end{equation}
where the operator in
\eqref{eq:scalar-stationary-K-resolvent}, initially defined on
\(D(K_\beta)\), is understood through its unique bounded extension to
\(L^2(\mathbb T)\).
\end{lemma}

\begin{proof}
Set
\begin{equation}\label{p_tau}
\mathcal P_\tau
=
L_\beta-\tau^2+i\tau a^2.
\end{equation}
The symbol of \(L_\beta\) is
\[
\ell_\beta(k)
=
\frac{k\tanh(\beta k)}{\beta}
=
\frac{|k|\tanh(\beta|k|)}{\beta}.
\]
Consequently,
\[
\beta L_\beta
=
|D|+R_\beta,
\]
where \(R_\beta\) is the Fourier multiplier with symbol
\[
r_\beta(k)
=
|k|
\left(
\tanh(\beta|k|)-1
\right).
\]
We have that
\[
\tanh y-1
=
O(e^{-2y}),
\qquad
y\to+\infty,
\]
thus
\[
r_\beta(k)
=
O_\beta
\left(
|k|e^{-2\beta|k|}
\right).
\]
Since \(r_\beta(k)\to0\) exponentially as \(|k|\to\infty\),
\(R_\beta\) is compact on \(L^2(\mathbb T)\). In particular, its
restriction defines a compact operator
\[
R_\beta:H^1(\mathbb T)\longrightarrow L^2(\mathbb T).
\]

Multiplying \eqref{p_tau} by \(\beta\), and setting
\[
\sigma=\sqrt{\beta}\,\tau,
\qquad
\chi_a=\sqrt{\beta}\,a^2,
\]
we obtain
\begin{equation}\label{eq:scalar-resolvent-reduction-D}
\beta\mathcal P_\tau
=
|D|-\sigma^2+i\sigma\chi_a+R_\beta.
\end{equation}
Consider the unperturbed stationary operator
\[
Q_\sigma
=
|D|-\sigma^2+i\sigma\chi_a.
\]
By hypothesis, 
\[
\chi_a\in C^\infty(\mathbb T),
\qquad
\chi_a\ge0,
\qquad
\chi_a\not\equiv0,
\]
therefore the high-frequency resolvent estimate in
\cite[Theorem~2]{AlazardMarzuolaWang} applies to \(Q_\sigma\), after a
change of sign in the spectral parameter.\footnote{More precisely,
Alazard, Marzuola, and Wang consider the stationary operator
\[
P(\sigma)=|D|-\sigma^2-i\sigma\chi
\]
and prove, for \(|\sigma|\) sufficiently large,
\[
\|P(\sigma)^{-1}\|_{\mathcal L(L^2(\mathbb T))}\le C.
\]
In the present setting,
\[
Q_\sigma=|D|-\sigma^2+i\sigma\chi_a
       =P(-\sigma)
\]
with \(\chi=\chi_a=\sqrt{\beta}\,a^2\) and
\(\sigma=\sqrt{\beta}\,\tau\). Thus their estimate applies directly to
\(Q_\sigma\).}
Hence there exist constants \(C>0\) and \(\sigma_0>0\) such that
\begin{equation}\label{eq:AMW-resolvent}
\|Q_\sigma^{-1}\|_{\mathcal L(L^2(\mathbb T))}
\le
C,
\qquad
|\sigma|\ge\sigma_0.
\end{equation}

We next show that the smoothing perturbation \(R_\beta\) preserves this
uniform estimate for sufficiently large \(|\sigma|\). Suppose, by
contradiction, that this is false. Then there exists a sequence
\[
|\sigma_n|\longrightarrow\infty
\]
and functions \(w_n\in H^1(\mathbb T)\) such that
\begin{equation}\label{eq:compact-perturbation-contradiction}
\|w_n\|_{L^2(\mathbb T)}
=
1
\end{equation}
and
\begin{equation}\label{eq:compact-perturbation-residual}
(Q_{\sigma_n}+R_\beta)w_n
=
f_n,
\qquad
f_n\longrightarrow0
\quad\text{in }L^2(\mathbb T).
\end{equation}
We first prove that
\begin{equation}\label{eq:wn-weak-zero}
w_n\rightharpoonup0
\qquad
\text{weakly in }L^2(\mathbb T).
\end{equation}
Let \(\varphi\in C^\infty(\mathbb T)\). Taking the
\(L^2\)-inner product of
\eqref{eq:compact-perturbation-residual} with \(\varphi\), we obtain
\[
-\sigma_n^2
\langle w_n,\varphi\rangle_{L^2}
=
\langle f_n,\varphi\rangle_{L^2}
-
\langle w_n,|D|\varphi\rangle_{L^2}
-
i\sigma_n
\langle w_n,\chi_a\varphi\rangle_{L^2}
-
\langle w_n,R_\beta^*\varphi\rangle_{L^2}.
\]
Hence
\begin{align*}
|\langle w_n,\varphi\rangle_{L^2}|
\le
\frac{
\|f_n\|_{L^2}\|\varphi\|_{L^2}
}{\sigma_n^2}
+
\frac{
\||D|\varphi\|_{L^2}
}{\sigma_n^2}
+
\frac{
\|\chi_a\varphi\|_{L^2}
}{|\sigma_n|}
+
\frac{
\|R_\beta^*\varphi\|_{L^2}
}{\sigma_n^2}.
\end{align*}
Since \(|\sigma_n|\to\infty\), the right-hand side tends to zero.
Therefore,
\[
\langle w_n,\varphi\rangle_{L^2}
\longrightarrow0
\]
for every \(\varphi\in C^\infty(\mathbb T)\). Since
\(\{w_n\}\) is bounded in \(L^2(\mathbb T)\) and
\(C^\infty(\mathbb T)\) is dense in \(L^2(\mathbb T)\),
\eqref{eq:wn-weak-zero} follows. We now use the fact that \(R_\beta\) is compact on \(L^2(\mathbb T)\),
\eqref{eq:wn-weak-zero} to ensure that
\begin{equation}\label{eq:Rbeta-wn-strong-zero}
R_\beta w_n
\longrightarrow0
\qquad
\text{strongly in }L^2(\mathbb T).
\end{equation}
From \eqref{eq:compact-perturbation-residual},
\[
Q_{\sigma_n}w_n
=
f_n-R_\beta w_n.
\]
Applying \eqref{eq:AMW-resolvent}, we obtain
\[
\|w_n\|_{L^2}
\le
C
\left(
\|f_n\|_{L^2}
+
\|R_\beta w_n\|_{L^2}
\right)
\longrightarrow0,
\]
thanks to \eqref{eq:Rbeta-wn-strong-zero}, contradicting \eqref{eq:compact-perturbation-contradiction}. 

Consequently, there exist constants \(C>0\) and \(\sigma_1>0\) such
that
\[
\|w\|_{L^2}
\le
C
\|
(Q_\sigma+R_\beta)w
\|_{L^2},
\qquad
|\sigma|\ge\sigma_1.
\]
By \eqref{eq:scalar-resolvent-reduction-D}, after adjusting the
constant, this gives
\eqref{eq:scalar-stationary-resolvent}.

For each fixed sufficiently large \(|\sigma|\), the resolvent estimate
\eqref{eq:AMW-resolvent} implies that
\[
Q_\sigma:H^1(\mathbb T)\longrightarrow L^2(\mathbb T)
\]
is an isomorphism, while
\[
R_\beta:
H^1(\mathbb T)
\longrightarrow
L^2(\mathbb T)
\]
is compact. Hence
\[
Q_\sigma+R_\beta:
H^1(\mathbb T)
\longrightarrow
L^2(\mathbb T)
\]
is Fredholm of index zero. The estimate above implies that its kernel
is trivial, and therefore it is surjective. Thus
\(Q_\sigma+R_\beta\), and consequently \(\mathcal P_\tau\), is
invertible for all sufficiently large \(|\tau|\).

We now prove \eqref{eq:scalar-stationary-K-resolvent}. From
\eqref{eq:scalar-stationary-resolvent},
\[
\|\mathcal P_\tau^{-1}\|_{\mathcal L(L^2)}
\le
C.
\]
Replacing \(\tau\) by \(-\tau\) gives the same estimate for
\[
\mathcal P_\tau^*
=
L_\beta-\tau^2-i\tau a^2.
\]
Let
\[
w
=
(\mathcal P_\tau^*)^{-1}\varphi,
\qquad
\varphi\in L^2(\mathbb T).
\]
Then
\[
\mathcal P_\tau^*w=\varphi.
\]
Taking the real part of the \(L^2\)-inner product with \(w\), and using
\[
L_\beta=K_\beta^2,
\]
we obtain
\[
\|K_\beta w\|_{L^2}^2
-
\tau^2\|w\|_{L^2}^2
=
\operatorname{Re}
\langle\varphi,w\rangle_{L^2}.
\]
Therefore,
\[
\|K_\beta w\|_{L^2}^2
\le
\tau^2\|w\|_{L^2}^2
+
\|\varphi\|_{L^2}\|w\|_{L^2}.
\]
Since
\[
\|w\|_{L^2}
\le
C\|\varphi\|_{L^2},
\]
we obtain, for \(|\tau|\ge1\),
\begin{equation}\label{eq:K-adjoint-resolvent}
\left\|
K_\beta
(\mathcal P_\tau^*)^{-1}\varphi
\right\|_{L^2}
\le
C|\tau|
\|\varphi\|_{L^2}.
\end{equation}
The operator \(K_\beta\) is self-adjoint and the operator
\(K_\beta(\mathcal P_\tau^*)^{-1}\) is bounded on \(L^2(\mathbb T)\), thus
we have
\[
\left(
K_\beta(\mathcal P_\tau^*)^{-1}
\right)^*
=
\mathcal P_\tau^{-1}K_\beta
\]
on \(D(K_\beta)\). Hence
\(\mathcal P_\tau^{-1}K_\beta\), initially defined on
\(D(K_\beta)\), extends uniquely to a bounded operator on
\(L^2(\mathbb T)\), and
\[
\|\mathcal P_\tau^{-1}K_\beta\|_{\mathcal L(L^2)}
=
\left\|
K_\beta
(\mathcal P_\tau^*)^{-1}
\right\|_{\mathcal L(L^2)}
\le
C|\tau|,
\]
where we have used \eqref{eq:K-adjoint-resolvent}. So, this proves \eqref{eq:scalar-stationary-K-resolvent}.
\end{proof}

\subsubsection*{\textbf{Step 3.} A high-frequency estimate for the second component}

The second damping term enters the transformed system through the
lower-order operator
\[
C_b
=
P_\beta^{1/2}b^2P_\beta^{1/2}.
\]
The following estimate provides the additional high-frequency gain
needed to control this contribution.

\begin{lemma}\label{lem:P-half-v-high-frequency}
There exist constants \(C>0\) and \(\tau_0>0\) such that every solution
of
\eqref{eq:L2-resolvent-first}--\eqref{eq:L2-resolvent-second}
satisfies
\begin{equation}\label{eq:P-half-v-estimate}
\|P_\beta^{1/2}v\|_{L^2(\mathbb T)}
\le
\frac{C}{|\tau|}
\left(
\|W\|_{L^2\times L^2}
+
\|G\|_{L^2\times L^2}
\right)
\end{equation}
whenever \(|\tau|\ge\tau_0\).
\end{lemma}

\begin{proof}
Let
\[
\Pi_\tau
=
\mathbf 1_{\{|K_\beta|\le|\tau|/2\}}
\]
be the spectral projector of \(K_\beta\). We let \(\Pi_\tau\) act
componentwise on \(L^2(\mathbb T)\times L^2(\mathbb T)\), and write
\[
W
=
W_{\mathrm{low}}
+
W_{\mathrm{high}},
\qquad
W_{\mathrm{low}}
=
\Pi_\tau W.
\]

Applying \(\Pi_\tau\) to
\eqref{eq:L2-resolvent-first}--\eqref{eq:L2-resolvent-second}, and
using the commutation of \(\Pi_\tau\) with \(K_\beta\), we obtain
\[
i
\begin{pmatrix}
\tau & K_\beta
\\
K_\beta & \tau
\end{pmatrix}
W_{\mathrm{low}}
=
\Pi_\tau G
-
\Pi_\tau
\begin{pmatrix}
a^2&0
\\
0&C_b
\end{pmatrix}
W.
\]
On the range of \(\Pi_\tau\),
\[
|\kappa_\beta(k)|
\le
\frac{|\tau|}{2},
\]
and therefore
\[
|\tau\pm\kappa_\beta(k)|
\ge
\frac{|\tau|}{2}.
\]
It follows that
\[
\left\|
\begin{pmatrix}
\tau & K_\beta
\\
K_\beta & \tau
\end{pmatrix}^{-1}
\Pi_\tau
\right\|_{\mathcal L(L^2\times L^2)}
\le
\frac{C}{|\tau|}.
\]
Since multiplication by \(a^2\) and the operator \(C_b\) are bounded
on \(L^2(\mathbb T)\), we obtain
\begin{equation}\label{eq:low-frequency-W}
\|W_{\mathrm{low}}\|_{L^2\times L^2}
\le
\frac{C}{|\tau|}
\left(
\|G\|_{L^2\times L^2}
+
\|W\|_{L^2\times L^2}
\right).
\end{equation}
Moreover, since \(P_\beta^{1/2}\) is bounded on \(L^2(\mathbb T)\),
\eqref{eq:low-frequency-W} gives
\begin{equation}\label{eq:P-half-v-low}
\|P_\beta^{1/2}\Pi_\tau v\|_{L^2}
\le
\frac{C}{|\tau|}
\left(
\|G\|_{L^2\times L^2}
+
\|W\|_{L^2\times L^2}
\right).
\end{equation}

It remains to estimate the complementary component. For \(k\neq0\),
\[
|\kappa_\beta(k)|^2
=
\frac{|k|\tanh(\beta|k|)}{\beta}
\le
\frac{|k|}{\beta}.
\]
On the support of \(I-\Pi_\tau\),
\[
|\kappa_\beta(k)|
>
\frac{|\tau|}{2},
\]
and hence
\[
|k|
>
\frac{\beta}{4}|\tau|^2.
\]
Moreover,
\[
p_\beta(k)^{1/2}
=
\left(
\frac{\tanh(\beta|k|)}{\beta|k|}
\right)^{1/2}
\le
\frac{1}{\sqrt{\beta|k|}}.
\]
Therefore, on the support of \(I-\Pi_\tau\),
\[
p_\beta(k)^{1/2}
\le
\frac{C_\beta}{|\tau|}.
\]
Consequently,
\begin{equation}\label{eq:P-half-v-high}
\left\|
P_\beta^{1/2}(I-\Pi_\tau)v
\right\|_{L^2}
\le
\frac{C}{|\tau|}
\|v\|_{L^2}.
\end{equation}

Combining
\eqref{eq:P-half-v-low} and
\eqref{eq:P-half-v-high}, we obtain
\[
\|P_\beta^{1/2}v\|_{L^2}
\le
\frac{C}{|\tau|}
\left(
\|W\|_{L^2\times L^2}
+
\|G\|_{L^2\times L^2}
\right),
\]
which proves \eqref{eq:P-half-v-estimate}.
\end{proof}

\subsubsection*{\textbf{Step 4.} The optimal resolvent upper bound}

We can now prove the main high-frequency resolvent estimate.

\begin{theorem}\label{thm:linear-resolvent-upper-bound}
There exist constants \(C>0\) and \(\tau_0>0\) such that
\begin{equation}\label{eq:linear-resolvent-upper-bound}
\left\|
(i\tau I-\mathcal A_d)^{-1}
\right\|_{\mathcal L(\mathcal H)}
\le
C|\tau|,
\qquad
|\tau|\ge\tau_0.
\end{equation}
\end{theorem}

\begin{proof}
By Corollary~\ref{cor:imaginary-axis-resolvent-damped},
\(i\mathbb R\subset\rho(\mathcal A_d)\). Thus, for
\(F=(f,g)^\top\in\mathcal H\), let
\[
Z=(h,u)^\top
=
(i\tau I-\mathcal A_d)^{-1}F
\in D(\mathcal A_d),
\]
and set
\[
v=P_\beta^{1/2}u,
\qquad
q=P_\beta^{1/2}g.
\]
We first derive a scalar equation for \(h\).

Multiplying \eqref{eq:L2-resolvent-first} by \(i\tau\), applying
\(iK_\beta\) to \eqref{eq:L2-resolvent-second}, and subtracting the
second resulting identity from the first, the terms containing
\(K_\beta v\) cancel. Since
\[
K_\beta^2=L_\beta,
\]
we obtain
\begin{equation}\label{eq:scalar-equation-h}
\left(
L_\beta-\tau^2+i\tau a^2
\right)h
=
i\tau f
-
iK_\beta q
+
iK_\beta C_bv.
\end{equation}
This identity is initially understood in \(H^{-1/2}(\mathbb T)\).
The term involving \(K_\beta q\) is handled through the bounded
extension of
\(\mathcal P_\tau^{-1}K_\beta\) established in
\eqref{eq:scalar-stationary-K-resolvent}.

Recall that
\[
C_b
=
P_\beta^{1/2}b^2P_\beta^{1/2}.
\]
Since \(K_\beta\) and \(P_\beta^{1/2}\) commute,
\[
K_\beta C_bv
=
K_\beta P_\beta^{1/2}
b^2P_\beta^{1/2}v.
\]
The operator \(K_\beta P_\beta^{1/2}\) is the Fourier multiplier with
symbol
\[
k p_\beta(k)
=
m_\beta(k)
=
\frac{\tanh(\beta k)}{\beta},
\]
which is bounded. Consequently,
\begin{equation}\label{eq:K-Cb-estimate}
\|K_\beta C_bv\|_{L^2}
\le
C
\|P_\beta^{1/2}v\|_{L^2}.
\end{equation}

Set
\[
\mathcal P_\tau
=
L_\beta-\tau^2+i\tau a^2.
\]
The term \(K_\beta q\) in \eqref{eq:scalar-equation-h} need not belong
to \(L^2(\mathbb T)\). We therefore interpret
\(\mathcal P_\tau^{-1}K_\beta q\) through the bounded extension
established in \eqref{eq:scalar-stationary-K-resolvent}. With this
interpretation, applying \(\mathcal P_\tau^{-1}\) to
\eqref{eq:scalar-equation-h} yields, in \(L^2(\mathbb T)\),
\[
h
=
i\tau\mathcal P_\tau^{-1}f
-
i\mathcal P_\tau^{-1}K_\beta q
+
i\mathcal P_\tau^{-1}K_\beta C_bv.
\]
Hence, by
\eqref{eq:scalar-stationary-resolvent},
\eqref{eq:scalar-stationary-K-resolvent}, and
\eqref{eq:K-Cb-estimate},
\begin{align}
\|h\|_{L^2}
&\le
C|\tau|\|f\|_{L^2}
+
C|\tau|\|q\|_{L^2}
+
C\|K_\beta C_bv\|_{L^2}
\notag\\
&\le
C|\tau|
\|G\|_{L^2\times L^2}
+
C\|P_\beta^{1/2}v\|_{L^2}.
\label{eq:h-upper-before-absorption}
\end{align}
By Lemma~\ref{lem:P-half-v-high-frequency},
\[
\|P_\beta^{1/2}v\|_{L^2}
\le
\frac{C}{|\tau|}
\left(
\|W\|_{L^2\times L^2}
+
\|G\|_{L^2\times L^2}
\right).
\]
Substituting this into
\eqref{eq:h-upper-before-absorption}, and enlarging the constant if
necessary, we obtain
\begin{equation}\label{eq:h-upper-with-W}
\|h\|_{L^2}
\le
C|\tau|
\|G\|_{L^2\times L^2}
+
\frac{C}{|\tau|}
\|W\|_{L^2\times L^2},
\qquad
|\tau|\ge\tau_0.
\end{equation}

By taking the \(L^2\)-inner product of
\eqref{eq:L2-resolvent-first} with \(h\), and of
\eqref{eq:L2-resolvent-second} with \(v\), taking imaginary parts,
and subtracting the resulting identities, we obtain
\[
\tau
\left(
\|h\|_{L^2}^2
-
\|v\|_{L^2}^2
\right)
=
\operatorname{Im}\langle f,h\rangle_{L^2}
-
\operatorname{Im}\langle q,v\rangle_{L^2}.
\]
Indeed, the damping terms give no contribution to the imaginary part,
once \(a^2\) and \(C_b\) are self-adjoint, while the terms involving
\(K_\beta\) cancel because \(K_\beta\) is self-adjoint. Hence
\[
|\tau|
\left|
\|h\|_{L^2}^2
-
\|v\|_{L^2}^2
\right|
\le
\|G\|_{L^2\times L^2}
\left(
\|h\|_{L^2}
+
\|v\|_{L^2}
\right).
\]
Therefore,
\[
\left|
\|h\|_{L^2}
-
\|v\|_{L^2}
\right|
\le
\frac{1}{|\tau|}
\|G\|_{L^2\times L^2},
\qquad
\tau\neq0,
\]
and consequently
\[
\|v\|_{L^2}
\le
\|h\|_{L^2}
+
\frac{1}{|\tau|}
\|G\|_{L^2\times L^2}.
\]
Thus,
\[
\|W\|_{L^2\times L^2}
\le
C\|h\|_{L^2}
+
\frac{C}{|\tau|}
\|G\|_{L^2\times L^2}.
\]
Combining this estimate with
\eqref{eq:h-upper-with-W}, we obtain
\[
\|W\|_{L^2\times L^2}
\le
C|\tau|
\|G\|_{L^2\times L^2}
+
\frac{C}{|\tau|}
\|W\|_{L^2\times L^2}.
\]
Choosing \(\tau_0\) sufficiently large so that
\[
\frac{C}{|\tau|}
\le
\frac12,
\qquad
|\tau|\ge\tau_0,
\]
the last term can be absorbed into the left-hand side. Hence
\[
\|W\|_{L^2\times L^2}
\le
C|\tau|
\|G\|_{L^2\times L^2},
\qquad
|\tau|\ge\tau_0.
\]

Finally, by the identities established in Step~1,
\[
\|W\|_{L^2\times L^2}
=
\|Z\|_{\mathcal H},
\qquad
\|G\|_{L^2\times L^2}
=
\|F\|_{\mathcal H}.
\]
Thus
\[
\|Z\|_{\mathcal H}
\le
C|\tau|
\|F\|_{\mathcal H}.
\]
Since
\[
Z
=
(i\tau I-\mathcal A_d)^{-1}F,
\]
we obtain
\eqref{eq:linear-resolvent-upper-bound}.
\end{proof}

Combining Theorem~\ref{thm:linear-resolvent-upper-bound} with the
quasimode lower bound, we obtain the sharpness of the polynomial
resolvent exponent.

\begin{corollary}
Assume, in addition, that there exists a nonempty open set $\mathcal O
\Subset
\mathbb T\setminus\operatorname{supp}(a)$.
Then the exponent \(1\) in the polynomial resolvent estimate is sharp.
More precisely,
\[
\left\|
(i\tau I-\mathcal A_d)^{-1}
\right\|_{\mathcal L(\mathcal H)}
=
O(|\tau|),
\qquad
\text{as }|\tau|\to\infty,
\]
while along the sequence \(\tau=\omega_N\),
\[
\left\|
(i\omega_NI-\mathcal A_d)^{-1}
\right\|_{\mathcal L(\mathcal H)}
\ge
c\,\omega_N.
\]
Consequently, no estimate of the form
\[
\left\|
(i\tau I-\mathcal A_d)^{-1}
\right\|_{\mathcal L(\mathcal H)}
=
O(|\tau|^\alpha),
\qquad
\text{as }|\tau|\to\infty
\]
can hold with \(\alpha<1\).
\end{corollary}

Figure~\ref{fig:spectral-structure} summarizes the spectral
picture relevant to the preceding resolvent analysis. The colored
points in the open left half-plane are purely schematic: they indicate
only that
\[
\sigma(\mathcal A_d)\subset\{\operatorname{Re}z<0\}
\]
and are not intended to represent the actual location of the damped
spectrum or any uniform spectral gap. The open circles on the imaginary
axis indicate the frequencies used in the quasimode construction.
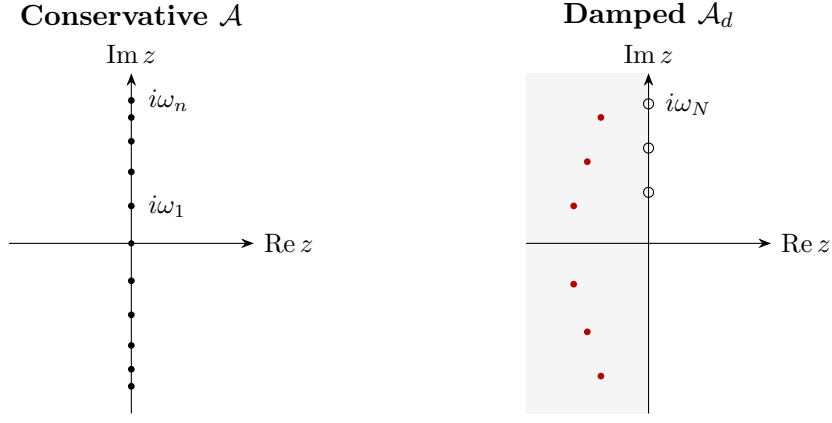
\begin{figure}[ht]
\centering
\begin{tikzpicture}[
    >=Stealth,
    scale=0.90,
    every node/.style={font=\small}
]


\begin{scope}[xshift=-3.8cm]

\draw[->,thin] (-1.8,0) -- (1.8,0)
    node[right] {$\operatorname{Re} z$};

\draw[->,thin] (0,-2.5) -- (0,2.5)
    node[above] {$\operatorname{Im} z$};

\fill (0,0) circle (1.3pt);

\foreach \y in {0.55,1.05,1.50,1.85,2.10}
{
    \fill (0,\y) circle (1.4pt);
    \fill (0,-\y) circle (1.4pt);
}

\node[right] at (0.10,0.55) {$i\omega_1$};
\node[right] at (0.10,2.10) {$i\omega_n$};

\node[font=\bfseries] at (0,3.35)
{Conservative $\mathcal{A}$};

\end{scope}


\begin{scope}[xshift=3.8cm]

\fill[gray!8] (-1.8,-2.5) rectangle (0,2.5);

\draw[->,thin] (-1.8,0) -- (1.8,0)
    node[right] {$\operatorname{Re} z$};

\draw[->,thin] (0,-2.5) -- (0,2.5)
    node[above] {$\operatorname{Im} z$};

\foreach \x/\y in {
    -1.10/0.55,
    -0.90/1.20,
    -0.70/1.85,
    -1.10/-0.60,
    -0.90/-1.30,
    -0.70/-1.95
}
{
    \fill[red!70!black] (\x,\y) circle (1.4pt);
}

\foreach \y in {0.75,1.40,2.05}
{
    \draw (0,\y) circle (2.1pt);
}

\node[right] at (0.10,2.05) {$i\omega_N$};

\node[font=\bfseries] at (0,3.35)
{Damped $\mathcal{A}_d$};

\end{scope}

\end{tikzpicture}

\caption{
Schematic spectral picture for the conservative and damped generators.
Left: $\sigma(\mathcal A)=\{0\}\cup\{\pm i\omega_n:n\in\mathbb N^*\}$,
with $\omega_n\sim\beta^{-1/2}n^{1/2}$.
Right: the colored points schematically indicate
$\sigma(\mathcal A_d)\subset\{\operatorname{Re}z<0\}$, while
$i\mathbb R\subset\rho(\mathcal A_d)$.
The open circles mark the quasimode frequencies $i\omega_N$.
}
\label{fig:spectral-structure}
\end{figure}

\section{Polynomial stabilization and optimal decay}
\label{sec:polynomial-stabilization}

We now combine the spectral and resolvent estimates obtained in the
previous sections to derive the decay rate of the damped linearized
Whitham--Boussinesq system. Recall that \(\mathcal A_d\) generates a contraction semigroup $\{S_d(t)\}_{t\ge0}$ on the Hilbert space \(\mathcal H\). Moreover,
Corollary~\ref{cor:imaginary-axis-resolvent-damped} gives $i\mathbb R \subset \rho(\mathcal A_d)$,  while Theorem~\ref{thm:linear-resolvent-upper-bound} shows that
\[
\left\|
(i\tau I-\mathcal A_d)^{-1}
\right\|_{\mathcal L(\mathcal H)}
=
O(|\tau|),
\qquad
\text{as }|\tau|\to\infty.
\]
We use the following characterization of polynomial stability.

\begin{theorem}[Borichev--Tomilov~\cite{BorichevTomilov2010}]
\label{thm:Borichev-Tomilov}
Let \(A\) be the generator of a bounded \(C_0\)-semigroup
\(\{S(t)\}_{t\ge0}\) on a Hilbert space \(H\), and assume that $i\mathbb R\subset\rho(A)$.
For every \(\alpha>0\), the following assertions are equivalent:
\[
\left\|
(isI-A)^{-1}
\right\|_{\mathcal L(H)}
=
O(|s|^\alpha),
\qquad
|s|\to\infty,
\]
and
\[
\left\|
S(t)A^{-1}
\right\|_{\mathcal L(H)}
=
O(t^{-1/\alpha}),
\qquad
t\to\infty.
\]
\end{theorem}

We can now prove the main result of the paper.

\begin{proof}[Proof of Theorem~\ref{thm:main-polynomial-stabilization}]
The operator \(\mathcal A_d\) generates a contraction semigroup on
\(\mathcal H\). Hence
\[
\sup_{t\ge0}
\|S_d(t)\|_{\mathcal L(\mathcal H)}
\le1.
\]
Furthermore,
Corollary~\ref{cor:imaginary-axis-resolvent-damped} yields $i\mathbb R\subset \rho(\mathcal A_d)$.  In particular, \(0\in\rho(\mathcal A_d)\), and therefore
\[
\mathcal A_d^{-1}\in\mathcal L(\mathcal H).
\]
By Theorem~\ref{thm:linear-resolvent-upper-bound},
\[
\left\|
(i\tau I-\mathcal A_d)^{-1}
\right\|_{\mathcal L(\mathcal H)}
\le
C|\tau|
\]
for all sufficiently large \(|\tau|\). Thus the resolvent condition in
Theorem~\ref{thm:Borichev-Tomilov} holds with \(\alpha=1\). Therefore,
\[
\left\|
S_d(t)\mathcal A_d^{-1}
\right\|_{\mathcal L(\mathcal H)}
\le
\frac{C}{t},
\qquad
t\ge1,
\]
giving \eqref{eq:semigroup-polynomial-decay}.

Let now \(Z_0\in D(\mathcal A_d)\). Since
\[
Z_0
=
\mathcal A_d^{-1}\mathcal A_dZ_0,
\]
we obtain
\begin{align*}
\|S_d(t)Z_0\|_{\mathcal H}
\le
\left\|
S_d(t)\mathcal A_d^{-1}
\right\|_{\mathcal L(\mathcal H)}
\|\mathcal A_dZ_0\|_{\mathcal H}
\le
\frac{C}{t}
\|\mathcal A_dZ_0\|_{\mathcal H}
\le
\frac{C}{t}
\|Z_0\|_{D(\mathcal A_d)},
\qquad
t\ge1,
\end{align*}
showing \eqref{eq:state-polynomial-decay} and \eqref{eq:state-polynomial-domain}.

By the definition of the energy inner product on \(\mathcal H\),
\[
E(t)
=
\frac12
\|h(t)\|_{L^2(\mathbb T)}^2
+
\frac12
\langle
P_\beta u(t),u(t)
\rangle_{H^{1/2},H^{-1/2}}
=
\frac12
\|Z(t)\|_{\mathcal H}^2.
\]
Consequently, \eqref{eq:energy-polynomial-decay} holds.

It remains to prove the optimality of the decay rate. Assume, in
addition, that the damping acting on the first component is genuinely
localized, that is, \eqref{eq:undamped-open-set} is satisfied. By Theorem~\ref{thm:resolvent-not-uniformly-bounded}, there exists a sequence $\omega_N\longrightarrow\infty$ such that
\[
\left\|
(i\omega_NI-\mathcal A_d)^{-1}
\right\|_{\mathcal L(\mathcal H)}
\ge
c\omega_N.
\]
Hence the resolvent cannot satisfy
\[
\left\|
(i\tau I-\mathcal A_d)^{-1}
\right\|_{\mathcal L(\mathcal H)}
=
O(|\tau|^\alpha),
\qquad
\text{as }|\tau|\to\infty
\]
for any \(\alpha<1\).

Suppose, by contradiction, that for some \(p>1\),
\[
\left\|
S_d(t)\mathcal A_d^{-1}
\right\|_{\mathcal L(\mathcal H)}
=
O(t^{-p}),
\qquad
\text{as }t\to\infty.
\]
Setting $\alpha=\frac1p<1$,
the converse implication in
Theorem~\ref{thm:Borichev-Tomilov} would give
\[
\left\|
(i\tau I-\mathcal A_d)^{-1}
\right\|_{\mathcal L(\mathcal H)}
=
O(|\tau|^\alpha),
\]
contradicting the lower bound above. Therefore, the rate
\[
\left\|
S_d(t)\mathcal A_d^{-1}
\right\|_{\mathcal L(\mathcal H)}
=
O(t^{-1})
\]
is optimal in the polynomial scale. The same argument also yields the optimality of the corresponding
energy decay. Indeed, since
\[
E(t)
=
\frac12
\|S_d(t)Z_0\|_{\mathcal H}^2,
\]
any uniform estimate of the form
\[
E(t)
\le
C t^{-q}
\|Z_0\|_{D(\mathcal A_d)}^2,
\qquad
q>2,
\]
would imply a uniform state decay rate of order
\(t^{-q/2}\), with \(q/2>1\), contradicting the optimality established
above. Hence the energy decay rate \(t^{-2}\) is also optimal in the
polynomial scale.
\end{proof}

\section{Final considerations}
\label{sec:final-considerations}

In this work, we established the optimal polynomial stabilization of the
linearized periodic Whitham--Boussinesq system in its natural energy
space. The analysis reflects the sublinear high-frequency dispersion of
the model and relies on a scalar reduction of the coupled resolvent
problem, together with the fractional-wave stationary estimate of
Alazard, Marzuola, and Wang~\cite{AlazardMarzuolaWang}. This yields the
linear high-frequency resolvent bound for the damped generator.

As a consequence of the Borichev--Tomilov theorem
\cite{BorichevTomilov2010}, solutions with initial data in
\(D(\mathcal A_d)\) satisfy
\[
\|S_d(t)Z_0\|_{\mathcal H}=O(t^{-1}),
\qquad
E(t)=O(t^{-2}),
\qquad
t\to\infty.
\]
The resolvent upper bound holds without the additional
genuine-localization assumption. When the damping acting on the first
component is genuinely localized, the high-frequency quasimodes
constructed above provide the matching lower bound, showing that these
decay rates are optimal in the polynomial scale. 
\subsection{Open issues and future perspectives}

The preceding results naturally lead to several questions concerning
the nonlinear dynamics, the influence of the damping geometry, and
possible extensions to other nonlocal dispersive systems. We briefly
discuss three directions below.

\medskip

\noindent\textbf{(i) Nonlinear stabilization.} The present analysis is restricted to the linearized dynamics. A natural
next question is whether the polynomial stabilization mechanism persists
for the nonlinear periodic Whitham--Boussinesq system. Addressing this
problem would require combining the linear decay estimates obtained here
with a well-posedness theory adapted to the nonlocal structure of the
full equations. The polynomial, rather than exponential, decay of the
linear semigroup makes the long-time treatment of the nonlinear
interactions particularly delicate.

In a different framework, recent work by Cavalcanti, Domingos
Cavalcanti, and Mu\~noz Rivera
\cite{CavalcantiCavalcantiMunoz2026} develops a resolvent-based approach
to polynomial decay for semigroups generated by maximal monotone
operators. Whether aspects of this philosophy can be adapted to the
Whitham--Boussinesq dynamics remains an interesting open question.

\medskip

\noindent\textbf{(ii) Geometry and strength of the damping.}
The optimality argument relies on the existence of a nonempty open
region where the damping acting on the first component vanishes. The
quasimodes constructed under this assumption rule out exponential
stability and show the sharpness of the polynomial resolvent exponent
in the genuinely localized setting. They do not, however, apply when
the damping acts everywhere on the torus.

It is therefore natural to investigate how the support, strength, and
vanishing profile of the damping affect the high-frequency resolvent.
In particular, it remains open whether stronger or nondegenerate
damping assumptions can improve the resolvent growth and lead to faster
decay, possibly even exponential stabilization.

\medskip

\noindent\textbf{(iii) Other Whitham-type and full-dispersion systems.}
The present analysis suggests that quantitative stabilization
properties are closely related to the order of the dispersion relation
and to the associated high-frequency structure of the underlying
Fourier multipliers. It would be interesting to determine whether
variants of the present strategy extend to other coupled Whitham-type,
full-dispersion, or fractional dispersive systems, and how the order of
the dispersive operator influences the optimal resolvent growth and the
corresponding decay rate.

\medskip

These questions point toward a broader direction in which dispersion,
coupling, and damping geometry interact in determining the
high-frequency resolvent behavior and the long-time decay of nonlocal
dispersive systems.

\subsection*{Acknowledgements}  R. de A. Capistrano–Filho was partially supported by CAPES/CO\-FE\-CUB grant number 88887.879175/2023-00, CNPq grant numbers 301744/2025-4 and 421573/2023-6, and PROPG (UFPE) \textit{via} PROAP resources.

The Article Processing Charge (APC) for the publication of this research was funded by the Coordenação de Aperfeiçoamento de Pessoal de Nível Superior - Brasil (CAPES) (ROR identifier: 00x0ma614).

\subsection*{Conflicts of Interest} The authors declare no conflicts of interest.
\subsection*{Data Availability Statement} No new data were generated.

\end{document}